\documentclass[11pt]{amsart}
\usepackage[leqno]{amsmath}
\usepackage{amssymb,latexsym,soul,cite,amsthm,color,enumitem,graphicx,mathtools,microtype,accents}
\usepackage[colorlinks=true,urlcolor=glaucous,citecolor=glaucous,linkcolor=glaucous,linktocpage,pdfpagelabels,bookmarksnumbered,bookmarksopen]{hyperref}
\definecolor{glaucous}{rgb}{0.38, 0.51, 0.71}
\usepackage{tikz}
\usepackage[english]{babel}
\usepackage[left=2.7cm,right=2.7cm,top=2.4cm,bottom=2.5cm]{geometry}

\numberwithin{equation}{section}

\newtheorem{theorem}{Theorem}[section]
\theoremstyle{plain}
\newtheorem{lemma}[theorem]{Lemma}
\theoremstyle{plain}
\newtheorem{proposition}[theorem]{Proposition}
\theoremstyle{plain}
\newtheorem{corollary}[theorem]{Corollary}

\theoremstyle{definition}
\newtheorem{remark}[theorem]{Remark}
\newtheorem{example}[theorem]{Example}

\newcommand{\N}{{\mathbb N}}

\newcommand{\R}{{\mathbb R}}
\newcommand{\eps}{\varepsilon}
\newcommand{\beq}{\begin{equation}}
\newcommand{\eeq}{\end{equation}}
\renewcommand{\le}{\leqslant}
\renewcommand{\ge}{\geqslant}

\newcommand{\w}{W^{s,p}_0(\Omega)}
\newcommand{\fpl}{(-\Delta)_p^s\,}

\newcommand{ \dd}{{\rm d}}
\newcommand{\ds}{{\rm d}^s_\Omega}
\newcommand{\p}{p^*_s}

\def\Xint#1{\mathchoice
{\XXint\displaystyle\textstyle{#1}}%
{\XXint\textstyle\scriptstyle{#1}}%
{\XXint\scriptstyle\scriptscriptstyle{#1}}%
{\XXint\scriptscriptstyle\scriptscriptstyle{#1}}%
\!\int}
\def\XXint#1#2#3{{\setbox0=\hbox{$#1{#2#3}{\int}$ }
\vcenter{\hbox{$#2#3$ }}\kern-.6\wd0}}
\def\fint{\Xint-}

\makeatletter
\newcommand{\leqnomode}{\tagsleft@true}
\newcommand{\reqnomode}{\tagsleft@false}
\makeatother

\newenvironment{enumroman}{\begin{enumerate}

}{\end{enumerate}}

\title[Fractional $p$-Laplacian with unbounded reactions]{On boundary regularity for the fractional $p$-Laplacian with unbounded reactions}

\author[A.\ Iannizzotto, S.\ Mosconi]{Antonio Iannizzotto, Sunra Mosconi}

\address[A.\ Iannizzotto]{Dipartimento di Matematica e Informatica
\newline\indent
Universit\`a degli Studi di Cagliari
\newline\indent
Via Ospedale 72, 09124 Cagliari, Italy}
\email{antonio.iannizzotto@unica.it}

\address[S.\ Mosconi]{Dipartimento di Matematica e Informatica
\newline\indent
Universit\`a degli Studi di Catania
\newline\indent
Viale A.\ Doria 6, 95125 Catania, Italy}
\email{mosconi@dmi.unict.it}


\begin{document}

\begin{abstract}
We consider an elliptic equation driven by the $s$-fractional $p$-Laplacian, set in a smooth bounded domain $\Omega\subset\R^N$ with homogeneous nonlocal Dirichlet conditions and a reaction $f$ lying in $L^q(\Omega)$ for some $q\ge 1$. We prove that the unique solution $u$ is $\alpha$-H\"older continuous up to the boundary, for any $\alpha$ below $p'(s-N/pq)$ if $N/ps<q\le N/s$, and $\alpha=s$ if $q>N/s$. Also, we prove that if $q>N/s$ then $u/\ds$ admits a H\"older continuous extension to the closure of $\Omega$, where $ \dd_\Omega$ denotes the distance from the boundary. Our results are almost optimal and extend previous regularity theorems known in  the linear case.
\end{abstract}

\maketitle

\begin{center}
\vspace{-1.5em}
Version of \today\
\end{center}

\section{Introduction}\label{sec1}

\noindent
The present paper is devoted to the study of global regularity of the solution of the following Dirichlet problem:
\beq\label{dir}
\begin{cases}
\fpl u = f & \text{in $\Omega$} \\
u = 0 & \text{in $\R^N\setminus\Omega$.}
\end{cases}
\eeq
The domain $\Omega\subset\R^N$ is assumed to be open, bounded, and with a $C^{1,1}$-smooth boundary $\partial\Omega$. Further, $0<s<1$, $p>1$, and the leading operator is the $s$-fractional $p$-Laplacian, defined as the differential of the convex, $C^1$-functional
\[u \mapsto \frac{1}{p}\iint_{\R^N\times\R^N}\frac{|u(x)-u(y)|^p}{|x-y|^{N+ps}}\,dx\,dy\]
in the fractional Sobolev space $\w$ (see Section \ref{sec2} below for details). Such operator is both nonlocal and nonlinear, referred to as degenerate if $p>2$ and singular if $1<p<2$, respectively (for $p=2$ it coincides with the fractional Laplacian). In special cases, it admits the heuristic alternative formulation
\[\fpl u(x) = 2\lim_{\eps\to 0^+}\,\int_{\R^N\setminus B_\eps(x)}\frac{|u(x)-u(y)|^{p-2}(u(x)-u(y))}{|x-y|^{N+ps}}\,dy.\]
Finally, the reaction is $f\in L^q(\Omega)$ with $q\ge 1$ fixed s.t.\ $f$ lies in the dual space of $\w$, which ensures well-posedness of problem \eqref{dir}. A classical variational argument shows that \eqref{dir} has a unique weak solution $u\in\w$.
\vskip2pt
\noindent
The problem of regularity of such solution has been the subject of a rapidly developing literature in recent years focused on deriving, under suitable additional assumptions on $f$, either higher (possibly fractional) differentiability properties of solutions, or their H\"older continuity. We aim at the latter setting and refer to \cite{BDLM, DKLN} for the most recent progresses on the former. Regarding the functional analytic framework for the forcing term, we will require that $f\in L^q(\Omega)$, and refer again to \cite{DKLN} and the literature therein for results requiring higher smoothness for $f$ instead. 

\subsection{Background and motivations}\label{ss11}

In order to appropriately introduce and justify the present investigation, let us briefly overview what is known in simpler frameworks. For $p=2$, $s=1$, \eqref{dir} essentially reduces to the following linear Dirichlet problem:
\[\begin{cases}
-\Delta u=f&\text{in $\Omega$},\\
u=0&\text{on $\partial\Omega$.} 
\end{cases}\]
The classical theory proves that, for $\Omega$ bounded and conveniently smooth and $f\in L^q(\Omega)$ with $q>N/2$ and $q\neq N$, we have the following sharp global regularity result (where we set $C^{1+\alpha}=C^{1,\alpha}$ for all $\alpha\in (0, 1)$):
\beq\label{lap}
f \in L^q(\Omega) \quad \Longrightarrow \quad u\in C^{2-\frac{N}{q}}(\overline\Omega)
\eeq
For the linear, nonlocal case ($p=2$, $0<s<1$) there is no complete analogue of \eqref{lap}, as the following example shows:

\begin{example}\label{tor}
Let $0<s<1$, $q>1$ be s.t.\ $0<2s-N/q\ne 1$, $f\in L^q_{\rm loc}(\Omega)$, and $u\in W^{s,2}_0(\Omega)$ solve
\beq\label{sfl}
\begin{cases}
(-\Delta)^s u = f & \text{in $\Omega$} \\
u = 0 & \text{in $\R^N\subset\Omega$.}
\end{cases}
\eeq
Then, $u$ belongs to the Bessel potential space $H^{2s,q}_{\rm loc}(\Omega)$. Recalling the standard Sobolev-Morrey embedding $H^{2s,q}(\R^N)\hookrightarrow C^{2s-\frac{N}{q}}(\R^N)$ and using the localization argument in \cite[Corollary 2.4, Lemma 2.9]{ROS}, we have $u\in C^{2s-N/q}_{\rm loc}(\Omega)$. In particular, if $s>1/2$ then $u\in C^1(\Omega)$ for all $q$ big enough, and if $s\le 1/2$ the limit regularity of $u$ for $q\to\infty$ is $C^{2s}$. However, no matter how summable $f$ is, the regularity of $u$ at the boundary cannot be better than $C^s$, corresponding to the threshold $q=N/s$. Indeed, the function
\[u(x) = (1-|x|^2)_+^s\]
(here and in the sequel $t_+=\max\{0, t\}$ for all $t\in \R$) has a constant $s$-fractional Laplacian in the unit ball $B_1(0)$, with $u\in C^\infty(B_1(0))$ but only $u\in C^s(\overline{B}_1(0))$.
\end{example}

\noindent
Note that the $C^s$ boundary regularity threshold outlined in Example \ref{tor} prevents the existence of a gradient on $\partial\Omega$, even when the solution is {\em a priori} $C^1$ in the interior (that is, when $2s-N/q>1$). Nevertheless, a convenient replacement of boundary gradient continuity can be retrieved in the fractional framework for all $s\in (0, 1)$, as suggested in the seminal paper \cite{ROS}. Assume that $\partial\Omega$ is $C^{1,1}$ and set for all $x\in\R^N$
\[\dd_\Omega(x) = {\rm dist}(x,\Omega^c) = \inf_{y\in\Omega^c}\,|x-y|.\]
If $u$ solves \eqref{sfl} with $f\in L^\infty(\Omega)$, then for any $0<\alpha<s$ the quotient $u/\ds$ turns out to have a $\alpha$-H\"older continuous extension to $\overline\Omega$, and and hence to all of $\R^N$ (see also \cite{ROS1}). Such property allows to replace  the gradient at the boundary as focus of interest with the fractional normal derivative, defined at $\bar x\in\partial\Omega$ by
\[\frac{\partial u}{\partial\nu^s}(\bar x) \sim \lim_{x\to\bar x}\,\frac{u(x)}{\ds(x)}.\]
The regularity properties at $\partial\Omega$ of the quotient above  is referred to as {\em fine boundary regularity}. The relevant extension to the case $f\in L^q(\Omega)$, $q>N/s$ is contained in \cite[Theorem 7.3]{Gr} where it is proved, among other things, that under such assumption (and identifying $u/\ds$ with its extension to $\overline\Omega$) the solution of \eqref{sfl} satisfies
\[\frac{u}{\ds} \in C^{s-\frac{N}{q}}(\overline{\Omega}).\]
Coupling the interior H\"older regularity seen in Example \ref{tor} and the boundary regularity of \cite{Gr} through an interpolation, we can summarize the linear case in the following scheme:
\beq\label{lin}
\begin{cases}
\displaystyle\frac{N}{2s} < q \le \frac{N}{s} \quad \Longrightarrow \quad u\in C^{2s-\frac{N}{q}}(\overline\Omega) \\
\displaystyle q > \frac{N}{s} \quad \Longrightarrow \quad u\in C^s(\overline\Omega), \ \frac{u}{\ds}\in C^{s-\frac{N}{q}}(\overline\Omega). \\
\end{cases}
\eeq
The aim of the present paper is to derive a statement parallel to \eqref{lin} for the nonlinear framework $p\neq 2$. In fact, the nonlinear setting is much more delicate and we will not be able to recover a full analogue with respect to fine boundary regularity, as the H\"older exponent of $u/\ds$ that we obtain is not explicitly determined. 
\vskip2pt
\noindent
Let us now review what is known in the nonlocal, nonlinear case ($0<s<1$, $p>1$). First, the condition $q>N/ps$ is naturally required to ensure {\em continuity} of solutions, as the following example shows:

\begin{example}\label{dis}
For $p=2$, the failure of the Bessel potential spaces embedding $H^{2s, N/2s}(\R^N)\hookrightarrow L^\infty_{\rm loc}(\R^N)$ implies via localization that there are unbounded functions whose fractional Laplacian belongs to  $L^{N/2s}(\Omega)$. \\ For $p\ne 2$,  $N>sp$ and $q<N/ps$, a simple example is the following. Let $\alpha$ satisfy
\[ \frac{1}{p-1}\Big(ps-\frac{N}{q}\Big) < \alpha < 0\]
and set for all $x\neq 0$
\[u_\alpha(x) = |x|^\alpha.\]
Then $u_\alpha\in W^{s,p}(B_1(0))$ and for a constant $C\neq 0$ depending on the parameters we have
\[\fpl u_\alpha(x) = C|x|^{(p-1)\alpha-ps}.\]
By the choice of exponents, we then have $\fpl u_\alpha\in L^q(B_1(0))$, but $u_\alpha$ even fails to be bounded in $B_1(0)$.
\end{example}

\noindent
In view of Example \ref{dis}, we will henceforth assume $q\ge 1$ and $q>N/ps$, and thus (setting as customary $p'=p/(p-1)$) define the positive threshold exponent
\beq\label{exp}
\bar\alpha = \min\Big\{1,\,p'\Big(s-\frac{N}{pq}\Big)\Big\}.
\eeq
In this connection, note that
\[\bar\alpha \le s \quad \Longleftrightarrow \quad q \le \frac{N}{s}.\]
The local H\"older regularity theory for \eqref{dir} started in \cite{DCKP} (see also \cite{CDI}), and subsequently developed up to the almost optimal H\"older continuity in \cite{BLS} (for $p>2$) and \cite{GL}   (for $1<p<2$). Summarizing these results, we have that the solution of \eqref{dir} lies in $C^\alpha_{\rm loc}(\Omega)$ for all $0<\alpha<\bar\alpha$.
\vskip2pt
\noindent
We recall {\em en passant} some further local regularity results, showing in particular that the optimal $C^{\bar\alpha}$ regularity is indeed achieved in many cases, provided $\bar\alpha<1$. In \cite{DN}, for the degenerate regime $p\ge 2$, $u\in C^{\bar\alpha}_{\rm loc}(\Omega)$ is obtained under the weaker assumption $f\in L^{q,\infty}_{\rm loc}(\Omega)$, $q>N/ps$. The case $f\in L^{p'}_{\rm loc}(\Omega)$, for any $p>1$ with $p'>N/ps$, can be treated through \cite{DKLN} and Sobolev-Morrey embedding. Similarly, the gradient estimates in \cite{BDLM} (see also \cite{GJS}) imply local $C^{\bar\alpha}$ regularity for any $p>1$ under the condition $p's>1$. Finally, in \cite{BT}, under the assumptions $p's\neq 1$ and $f\in L^\infty(\Omega)$, $C^{\bar\alpha}$ regularity is proved even when $\bar\alpha=1$ (see \cite[Example 1.6]{BLS}). Unfortunately, though, the picture is not complete yet.
\vskip2pt
\noindent
Regarding boundary regularity, the optimal rate of continuity $u\in C^s(\overline\Omega)$ (hence independent of $p$) was obtained in \cite{IM} for $f\in L^\infty(\Omega)$ along with the fine boundary regularity $u/\ds\in C^\alpha(\overline\Omega)$ with an undetermined $\alpha\le s$ (partial results were obtained in \cite{IMS,IMS1}). Such amount of global regularity allows for nice applications in existence theory, comparison principles, bifurcation results and so on, see for instance \cite{I} for a brief account on the subject.

\subsection{Main results}\label{ss12}

Our contributions intervene at this point. Regarding pure H\"older regularity, we shall prove the following result which extends local continuity to the boundary, with a uniform norm estimate:

\begin{theorem}\label{hol}
{\em (Global regularity)} Let $\Omega\subset\R^N$ be open bounded with a $C^{1,1}$-boundary, $p>1$, $0<s<1$, $f\in L^q(\Omega)$ with $q\ge 1$, $q>N/ps$ and $u\in\w$ be the solution of \eqref{dir}. If $\bar\alpha$ is given by \eqref{exp} then for any $0<\alpha\le s$ s.t.\ $\alpha<\bar\alpha$ there exists $C_\alpha>0$ depending on $N,p,s,\Omega,q$, and $\alpha$, s.t.\
\[\|u\|_{C^\alpha(\overline\Omega)} \le C_\alpha\|f\|_{L^q(\Omega)}^\frac{1}{p-1}.\]
\end{theorem}

\noindent
Some comments on Theorem \ref{hol} above are in order. First note that, as already recalled, the assumption $q>N/ps$ forces $L^q(\Omega)$ to be contained in the dual of $\w$, so that problem \eqref{dir} is well posed. Also, recalling that $\bar\alpha\le s$ iff $q\le N/s$, and in light of Example \ref{dis}, the previous statement determines two different regimes for $q$:
\begin{itemize}[leftmargin=1cm]
\item[$(a)$] When $q\in (N/ps, N/s]$, there is no difference between up-to-the boundary and interior H\"older regularity (in the sense mentioned before). Furthermore, the corresponding exponent is almost optimal, see Example \ref{opt} below.
\item[$(b)$] When $q\in (N/s, \infty)$ the optimal (due to Example \ref{tor}) $C^s$ regularity up to the boundary is achieved.
\end{itemize}
In particular, the H\"older regularity up to the boundary in the nonlinear setting turns out to be completely analogous to the previously described linear case. A clear drawback of Theorem \ref{hol} is the lack of global $C^{\bar\alpha}$ regularity in case $(a)$ above. A cleaner and stronger statement, which we conjecture to be true, would be
\[f\in L^q(\Omega)\quad \Longrightarrow \quad u\in C^{\min\{s, \bar\alpha\}}(\overline{\Omega})\]
whenever $\bar\alpha>0$. More precisely, we believe that whenever $q\in(N/ps,N/s]$ and interior $C^{\bar\alpha}$ regularity holds true for local solutions of \eqref{dir}, the same regularity persists up to the boundary. Some of such cases have been reviewed in Subsection \ref{ss11}.
\vskip2pt
\noindent
Note that, despite the nonlocal nature of the operator, which may in principle improve H\"older regularity near $\partial\Omega$, our result is almost optimal, also in term of boundary regularity, as the following example shows:

\begin{example}\label{opt}
(Almost optimal boundary regularity) For simplicity, we consider the one-dimensional case, setting $N=1$, $\Omega=(0,\infty)$, and choosing $q\in(1/ps,1/s)$, which entails
\[\bar\alpha = p'\Big(s-\frac{1}{pq}\Big) \in (0,s).\]
We want to show that $C^\beta(\overline\Omega)$ regularity cannot hold in general, for any $\beta>\bar\alpha$. Indeed, pick $\alpha\in(\bar\alpha,s)$ and set for all $x\in\R$
\[u_\alpha(x) = x_+^\alpha.\]
Then, we clearly have $u_\alpha\in C^\alpha(\overline\Omega)\setminus C^\beta(\overline\Omega)$, while for all $x\in\Omega$
\[\fpl u_\alpha(x) = Cx^{(p-1)\alpha-ps}\]
for a constant $C\neq 0$ depending on the data, implying $\fpl u_\alpha\in L^q_{\rm loc}(\Omega)$ as
\[q < \frac{1}{ps-(p-1)\alpha}.\]
Similarly, $u_{\bar\alpha}\in C^{\bar\alpha}(\overline\Omega)$ but $\fpl u_{\bar\alpha}\notin L^q_{\rm loc}(\Omega)$. See Appendix \ref{app} for an example in dimension $N\ge 2$ (for $p=2$).
\end{example}

\noindent
Regarding fine (or weighted) boundary regularity, we present the following extension of the known result to the case of unbounded reactions:

\begin{theorem}\label{fbr}
{\rm (Fine boundary regularity)} Let $\Omega\subset\R^N$ be open bounded with a $C^{1,1}$-boundary, $p>1$, $0<s<1$, $f\in L^q(\Omega)$ with $q>N/s$, and $u\in\w$ be the solution of \eqref{dir}. Then, there exist $\alpha\in(0,s]$, $C_\alpha>0$ depending on $N,p,s,\Omega$, and $q$, s.t.\ $u/\ds$ admits a $\alpha$-H\"older continuous extension to $\overline\Omega$ and
\[\Big\|\frac{u}{\ds}\Big\|_{C^\alpha(\overline\Omega)} \le C_\alpha\|f\|_{L^q(\Omega)}^\frac{1}{p-1}.\]
\end{theorem}

\noindent
Again the range of $q$ considered in the previous statement is almost optimal for the continuity of $u/\ds$. Indeed, looking at Example \ref{opt} we see that already in one dimension and for $q\in (1/ps, 1/s)$, the function $u_\alpha$ considered there is s.t.\ $u_\alpha/\ds$ is not even bounded. 

\subsection{Sketch of proof}\label{ss13}

Let us briefly discuss the proof of our main results. To give a general overview, the case $q\in (N/ps, N/s]$ of Theorem \ref{hol} can be proved independently. This in turn implies that, when $q>N/s$, $u$ is $C^\alpha(\overline{\Omega})$ for any $\alpha<s$. In order to prove $u\in C^s(\overline{\Omega})$ for $q>N/s$, we can employ an interpolation argument as in \cite[Theorem 2.7]{IM}, based on interior regularity and an additional $L^\infty(\Omega)$ bound on $u/\ds$. Since the latter is an endpoint estimate, it comes with no surprise that it is actually easier to prove the full Theorem \ref{fbr} independently of Theorem \ref{hol}. Then the $L^\infty$-estimate is a consequence of Theorem \ref{fbr} which allows to  complete the proof of Theorem \ref{hol} in the range $q>N/s$. 
\vskip2pt
\noindent
Regarding the techniques, our approach is based on Campanato's characterization of H\"older continuity in terms of integral oscillations, allowing to consider, instead of  the point-wise oscillation
\[\underset{B_r(x_0)}{{\rm osc}}\,u = \sup_{B_r(x_0)} u-\inf_{B_r(x_0)} u\]
the more convenient mean oscillation (or $p$-{\em variance})
\[\sigma(u, x_0, r) = \fint_{B_r(x_0)} |u-(u)_{x_0, r}|^p\, dx, \quad (u)_{x_0, r} = \fint_{B_{r}(x_0)} u\, dx.\]
This is coupled with  the classical technique of deriving good estimates on the $(s,p)$-harmonic extensions of $u$ on arbitrary balls, namely the solutions of 
\[\begin{cases}
\fpl v = 0 & \text{in $B_R(x_0)\cap \Omega$}\\
v = u & \text{in $\R^N\setminus(B_R(x_0)\cap\Omega)$,}
\end{cases}\]
with $x_0\in \overline{\Omega}$ and $R<{\rm diam}(\Omega)$. With respect to well documented interior estimates already present in the literature, we encounter here several new difficulties. The nonlocal nature of the equation forces the local controls on $v$ near the boundary to depend on the global behaviour of $u$ in $\overline\Omega$.  More precisely, we obtain a family of oscillation estimates on $v$, each one depending on an {\em a priori} bound on $[u]_\beta$ (where the latter stands for the seminorm of $u$ in $C^\beta(\overline\Omega)$). Then in the standard scheme  
\[\sigma(u, x_0, r) \lesssim \sigma(u-v, x_0, r) +\sigma(v, x_0, r),\]
the first  term is controlled by $\|f\|_{L^q(\Omega)}$ via monotonicity properties of $\fpl$ (particularly involved in the singular case $p<2$, see Lemma \ref{mos}), while the second one is iteratively controlled through $[u]_\beta$. The resulting estimate is of the form
\[[u]_{\tilde\beta} \lesssim  \|f\|_{L^q(\Omega)}^\frac{1}{p-1}+[u]_\beta,\]
for an explicit $\tilde{\beta}>\beta$. The latter can therefore be iterated from $\beta=0$ (where we can use the $L^\infty(\Omega)$ norm instead of the $C^\beta$ seminorm) to reach any $\alpha<\bar\alpha$ when $q\in (N/ps, N/s]$, thus proving Theorem \ref{hol} in this regime.
\vskip2pt
\noindent
As already mentioned, we  then turn to the proof of Theorem \ref{fbr}. By approximation we can qualitatively assume $f\in L^\infty(\Omega)$ so that, thanks to \cite{IMS1, IM}, $u/\ds$ is continuous. For the corresponding $(s, p)$-harmonic extensions $v$ we thus derive local controls on $v/\ds$ near $\partial\Omega$ depending on the global behaviour of $u/\ds$. Here the barrier arguments in \cite{IMS} and the fine boundary regularity results in \cite{IM, IMS1} play a major role.  Then, in order to transfer the regularity of $v/\ds$ on $u/\ds$, we use a refined form of monotonicity of $\fpl$ through the fractional Hardy inequality (see Theorem \ref{fhi}). The resulting estimate for all $q>N/s$ is of the form
\[\Big[\frac{u}{\ds}\Big]_{\alpha} \lesssim \|f\|_{L^q(\Omega)}^\frac{1}{p-1}+\Big\|\frac{u}{\ds}\Big\|_{L^\infty(\Omega)},\]
for some positive but otherwise unspecified $\alpha$. A standard argument allows to reabsorb the last term, thus giving the sought estimate and completing the proof of Theorem \ref{fbr}. We finally implement the above mentioned interpolation argument to complete the proof of Theorem \ref{hol} for $q>N/s$.

\subsection{Structure of the paper and notations}\label{ss14}

The structure of the paper is the following: in Section \ref{sec2} we recall some basic notions on the fractional $p$-Laplacian, including local regularity, boundary behavior, Hardy's inequality and monotonicity. Section \ref{sec3} is devoted to regularity estimate at the boundary of $(s, p)$-harmonic extensions of a given function; despite the logical order of the previous discussion, in Section \ref{sec4} we prove Theorem \ref{fbr} first; and in Section \ref{sec5} we prove Theorem \ref{hol}. Finally, Appendix \ref{app} is devoted to complete Example \ref{opt}.
\vskip2pt
\noindent
For all $x_0\in\R^N$, $r>0$ we denote by $B_r(x_0)$ the open ball centered at $x_0$ with radius $r$, and we assume $x_0=0$ when the center is omitted. For all $U\subset\R^N$ we set $U^c=\R^N\setminus U$, we denote by $|U|$ the $N$-dimensional Lebesgue measure of $U$ and with $\omega_N$ the volume of the unit ball of $\R^N$. For $U\subseteq\R^N$, we set
\[ \dd_U(x) = {\rm dist}(x,U^c) = \inf_{y\in U^c}\,|x-y|.\]
For all function $u$ defined in $U$ we will still denote by $u$ its extension  to the whole $\R^N$ as $0$ on $U^c$. We say that $u>v$ in $U$ if $u(x)>v(x)$ for a.e.\ $x\in U$ (and similar relations). We denote the essential supremum, infimum, and oscillation of $u$ in $U$ by 
\[\sup_U\,u ,\quad \inf_U u, \quad \underset{U}{\rm osc}\,u = \sup_U\,u-\inf_U\,u,\]
 respectively. For all $a\in\R$ and $p>1$, we will set for short $a^{p-1}=|a|^{p-2}a$. Finally, we refer to $N,p,s,\Omega$ as the {\em data}, and denote by $C$ several positive constants depending on the data.

\section{Preliminaries}\label{sec2}

\noindent
We begin by recalling some definitions about fractional Sobolev spaces, referring to \cite{L} for a complete account on the subject. First, for any measurable $u:\R^N\to\R$ we define the Gagliardo seminorm
\[[u]_{s,p} = \left[\iint_{\R^N\times\R^N}\frac{|u(x)-u(y)|^p}{|x-y|^{N+ps}}\,dx\,dy\right]^\frac{1}{p},\]
with $0<s<1$, $p>1$. We say that $u\in W^{s,p}(\R^N)$ if $u\in L^p(\R^N)$ and $[u]_{s,p}<\infty$. The definition of $W^{s,p}(\Omega)$ for an open $\Omega\subset\R^N$ is analogous, with the  integral of the seminorm restricted to $\Omega\times\Omega$. For a bounded domain $\Omega\subset\R^N$, we denote by $\w$ the subspace containing all $u\in W^{s,p}(\R^N)$ s.t.\ $u=0$ in $\Omega^c$. The space $\w$, endowed with the norm $[u]_{s,p}$, is a uniformly convex, separable Banach space with dual $W^{-s,p'}(\Omega)$, and it is compactly embedded into $L^q(\Omega)$ for all $q\in[1,\p)$, where
\[\p = \begin{cases}
\displaystyle\frac{Np}{N-ps} & \text{if $N>ps$} \\
\infty & \text{if $N<ps$}
\end{cases}\]
while the embedding into $L^{\p}(\Omega)$ is continuous if $N>ps$. For a bounded $\Omega$, we denote by $\widetilde{W}^{s,p}(\Omega)$ the space of all $u\in L^p_{\rm loc}(\R^N)$ s.t.\ $u\in W^{s,p}(\Omega')$ for some $\Omega'\Supset\Omega$, and
\[\int_{\R^N}\frac{|u(x)|^{p-1}}{(1+|x|)^{N+ps}}\,dx < \infty.\]
The latter is a natural framework for a rigorous definition of the $s$-fractional $p$-Laplacian as an operator $\fpl:\widetilde{W}^{s,p}(\Omega)\to W^{-s,p'}(\Omega)$. For all $u\in\widetilde{W}^{s,p}(\Omega)$, $\varphi\in\w$ we set
\[\langle\fpl u,\varphi\rangle = \iint_{\R^N\times\R^N}\frac{(u(x)-u(y))^{p-1}(\varphi(x)-\varphi(y))}{|x-y|^{N+ps}}\,dx\,dy\]
(with $a^{p-1}=|a|^{p-2}a$). Such definition agrees with the one given in Section \ref{sec1}, and is in general equivalent to another definition frequently used in the literature, based on tail spaces (see for instance \cite{BLS}). We recall that $\fpl$, restricted to $\w$, satisfies the $(S)_+$-condition: that is, whenever $u_n\rightharpoonup u$ in $\w$ and
\[\limsup_n\,\langle\fpl u_n,u_n-u\rangle \le 0,\]
then $u_n\to u$ in $\w$ (see \cite[Lemma 2.1]{FI}). Now let $f\in L^q(\Omega)$ for some $q\ge (p^*_s)'$. We say that $u\in\widetilde{W}^{s,p}(\Omega)$ is a (weak) solution of the equation
\beq\label{fee}
\fpl u = f \qquad  \text{ in $\Omega$,}
\eeq
if for all $\varphi\in\w$
\[\langle\fpl u,\varphi\rangle = \int_\Omega f\varphi\,dx,\]
noting that the identity above is well posed. In particular, if $u\in\w$ solves \eqref{fee}, then we say that $u$ is a (weak) solution of \eqref{dir}. Sub- and supersolutions are meant in an analogous weak sense.
\vskip2pt
\noindent
We next recall two useful technical properties of $\fpl$. First, a weak comparison principle from \cite[Proposition 2.1]{IMS1}:

\begin{lemma}\label{wcp}
{\rm (Weak comparison)} Let $u,v\in\widetilde{W}^{s,p}(\Omega)$ be s.t.\
\[\begin{cases}
\fpl u \le \fpl v & \text{in $\Omega$} \\
u \le v & \text{in $\Omega^c$.}
\end{cases}\]
Then, $u\le v$ in $\R^N$.
\end{lemma}

\noindent
Then, a nonlocal superposition principle, slightly rephrased from \cite[Proposition 2.6]{IMS1}:

\begin{lemma}\label{spp}
{\rm (Nonlocal superposition)} Let $\Omega$ be bounded, $u\in\widetilde{W}^{s,p}(\Omega)$, $v\in L^1_{\rm loc}(\R^N)$ s.t.\ ${\rm supp}\, (v-u)\subseteq \overline{\Omega}^c$ and
\[\int_{{\rm supp}\, (v-u)} \frac{|v(y)|^{p-1}}{(1+|y|)^{N+ps}}\,dy < \infty.\]
 Then $v\in\widetilde{W}^{s,p}(\Omega)$ and, weakly in $\Omega$, there holds 
\[\fpl v(x) = \fpl u(x)+2\int_{\R^N}\frac{(u(x)-v(y))^{p-1}-(u(x)-u(y))^{p-1}}{|x-y|^{N+ps}}\,dy.\]
\end{lemma}

\subsection{Regularity theory}\label{ss21}

We recall some known results about local regularity of the solutions of equation \eqref{fee}, already mentioned in Subection \ref{ss11}. In such results, a fundamental notion is that of {\em nonlocal tail}, which for all $x_0\in\R^N$, $R>0$ is defined by
\beq\label{tail}
{\rm Tail}(u,x_0,R) = \left[R^{ps}\int_{B_R^c(x_0)}\frac{|u(x)|^{p-1}}{|x-x_0|^{N+ps}}\,dx\right]^\frac{1}{p-1}.
\eeq
In this connection, $\Omega$ is a general open subset of $\R^N$, and as usual $p>1$, $0<s<1$. We begin with a priori bounds from \cite[Theorem 3.2]{BLS}:

\begin{theorem}\label{apb}
{\rm (A priori bound)} Let $f\in L^q(\Omega)$ with $q\ge 1$, $q>N/ps$, $B_{2R}(x_0)\subseteq\Omega$, and $u\in\widetilde{W}^{s,p}(\Omega)$ be a solution of \eqref{fee}. Then, $u\in L^\infty(B_{R/2}(x_0))$ and there exists $C>0$ depending on the data and $q$, s.t.\
\[\|u\|_{L^\infty(B_{R/2}(x_0))} \le C\left[\Big(\fint_{B_R(x_0)}|u|^p\,dx\Big)^\frac{1}{p}+\Big(R^{ps-\frac{N}{q}}\|f\|_{L^q(B_R(x_0))}\Big)^\frac{1}{p-1}+{\rm Tail}(u,x_0,R)\right].\]
\end{theorem}

\noindent
The following local H\"older continuity result, proved in \cite[Theorem 1.4]{BLS} (for the degenerate regime) and in \cite[Theorem 1.2]{GL} (for the singular regime), is here rephrased with a slight modification of radii:

\begin{theorem}\label{lhr}
{\rm (Local regularity)} Let $f\in L^q(\Omega)$ with $q\ge 1$, $q>N/ps$, $B_{2R}(x_0)\subseteq\Omega$, and $u\in\widetilde{W}^{s,p}(\Omega)$ be a solution of \eqref{fee}. Also let $0<\delta<\bar\alpha$, with $\bar\alpha$ defined by \eqref{exp}. Then, $u\in C^\delta(\overline{B}_{R/2}(x_0))$ and there exists $C_\delta>0$ depending on the data, $q$, and $\delta$, s.t.\
\[[u]_{C^\delta(\overline{B}_{R/2}(x_0))} \le \frac{C_\delta}{R^\delta}\left[\|u\|_{L^\infty(B_R(x_0))}+\Big(R^{ps-\frac{N}{q}}\|f\|_{L^q(B_R(x_0))}\Big)^\frac{1}{p-1}+{\rm Tail}(u,x_0,R)\right].\]
\end{theorem}

\noindent
Note that the original statement of Theorem \ref{lhr} requires $u$ to be locally bounded, but this is in fact ensured by Theorem \ref{apb} above. For future use, we point out the following consequence of Theorem \ref{lhr}. For all $\alpha \in [0, 1)$ define the (possibly infinite) quantity
\beq\label{hal} H_\alpha(u) = \begin{cases}
\|u\|_{L^\infty(\Omega)} & \text{if $\alpha=0$} \\
[u]_{C^\alpha(\overline\Omega)} & \text{if $\alpha>0$.}
\end{cases}
\eeq
We can remove the tail term from the estimate of Theorem \ref{lhr} as follows:

\begin{corollary}\label{chr}
Let $f\in L^q(\Omega)$ with $q\ge 1$, $q>N/ps$, $B_{2R}(x_0)\subseteq\Omega$, $u\in\widetilde{W}^{s,p}(\Omega)$ be a solution of \eqref{fee}, $0<\delta<\bar\alpha$, with $\bar\alpha$ defined by \eqref{exp}, $0\le\alpha<\min\{1,p's\}$, and $H_\alpha(u)$ be defined by \eqref{hal}. Then, there exists $C_\alpha>0$ depending on the data, $q$, $\delta$, and $\alpha$, s.t.\
\[[u]_{C^\delta(\overline{B}_{R/2}(x_0))} \le \frac{C_\alpha}{R^\delta}\left[H_\alpha(u) R^\alpha+\Big(R^{ps-\frac{N}{q}}\|f\|_{L^q(B_R(x_0))}\Big)^\frac{1}{p-1}\right].\]
\end{corollary}
\begin{proof}
We apply Theorem \ref{lhr} to $u-u(x_0)\in\widetilde{W}^{s,p}(\Omega)$, which still solves \eqref{fee}. Fix $\alpha\in[0,p's)$, and without loss of generality assume $H_\alpha(u)<\infty$. Clearly we have
\[\|u-u(x_0)\|_{L^\infty(B_R(x_0))} \le 2H_\alpha(u) R^\alpha.\]
To estimate the tail, assume first $\alpha=0$ and $u$ bounded:
\[\int_{B_R^c(x_0)}\frac{|u(x)-u(x_0)|^{p-1}}{|x-x_0|^{N+ps}}\,dx \le \int_{B_R^c(x_0)}\frac{(2H_0(u))^{p-1}}{|x-x_0|^{N+ps}}\,dx = \frac{CH_0^{p-1}(u)}{R^{ps}},\]
which by definition \eqref{tail} implies ${\rm Tail}(u,x_0,R)\le CH_0(u)$. If $\alpha>0$ and $u\in C^\alpha(\R^N)$ (note that if $u\in C^\alpha(\overline\Omega)$ then $u$ can be extended to $\R^N$ with the same H\"older seminorm), by $ps-(p-1)\alpha>0$ we have
\[\int_{B_R^c(x_0)}\frac{|u(x)-u(x_0)|^{p-1}}{|x-x_0|^{N+ps}}\,dx \le \int_{B_R^c(x_0)}\frac{H_\alpha^{p-1}(u)|x-x_0|^{(p-1)\alpha}}{|x-x_0|^{N+ps}}\,dx \le \frac{CH_\alpha^{p-1}(u)}{R^{ps-(p-1)\alpha}},\]
which again implies ${\rm Tail}(u,x_0,R)\le CH_\alpha(u) R^\alpha$. In either case we obtain
\[[u]_{C^\delta(B_{R/2}(x_0))} \le \frac{C_\alpha}{R^\delta}\left[H_\alpha(u) R^\alpha+\Big(R^{ps-\frac{N}{q}}\|f\|_{L^q(B_R(x_0))}\Big)^\frac{1}{p-1}\right],\]
with $C_\alpha\to\infty$ as $\alpha\to p's$.
\end{proof}

\noindent
We now examine regularity up to the boundary for solutions of the Dirichlet problem \eqref{dir}. In the case of {\em bounded} reactions, this amounts to the optimal $C^s$ regularity (see Example \ref{tor}), and in addition the quotient $u/\ds$ is H\"older continuous with an undetermined exponent (see \cite[Theorems 1.1, 2.7]{IM}), where we recall that $ \dd_\Omega$ denotes the distance function from $\partial\Omega$.

\begin{theorem}\label{bhr}
{\rm (Boundary regularity)} Let $\Omega$ be bounded with a $C^{1,1}$-boundary, $f\in L^\infty(\Omega)$, $u\in\w$ be the solution of \eqref{dir}. Then:
\begin{enumroman}
\item\label{bhr1} $u\in C^s(\overline\Omega)$;
\item\label{bhr2} there exist $0<\alpha\le s$, $C_\alpha>0$ depending on the data, s.t.\ $u/\ds$ admits a $\alpha$-H\"older continuous extension to $\overline\Omega$ and
\[\Big\|\frac{u}{\ds}\Big\|_{C^\alpha(\overline\Omega)} \le C_\alpha\|f\|_{L^\infty(\Omega)}^\frac{1}{p-1}.\]
\end{enumroman}
\end{theorem}

\noindent
We also recall a technical lemma, which is widely used in the fractional literature to shift interior H\"older regularity to the boundary (see \cite[Lemma 2.6]{IM} and \cite[proof of Theorem 4.5]{ROS}):

\begin{lemma}\label{rsb}
Let $\Omega$ be bounded with a $C^{1,1}$-boundary, $u\in L^\infty(\Omega)$, $\beta\in(0,1)$, $M,\nu>0$ s.t.\
\begin{enumroman}
\item\label{rsb1} $\|u\|_{L^\infty(\Omega)} \le M$;
\item\label{rsb2} for all $x_0\in\Omega$, $R>0$ s.t.\ $ \dd_\Omega(x_0)=4R$ there holds $u\in C^\beta(\overline{B}_{R/8}(x_0))$ with
\[[u]_{C^\beta(\overline{B}_{R/8}(x_0))} \le M\Big(1+\frac{1}{R^\nu}\Big);\]
\item\label{rsb3} for all $\bar x\in\partial\Omega$, $r>0$ small enough
\[\underset{\Omega\cap B_r(\bar x)}{\rm osc}\,u \le Mr^\beta.\]
\end{enumroman}
Then, $u\in C^\alpha(\overline\Omega)$ and $[u]_{C^\alpha(\overline\Omega)}\le C$, with $C>0$ depending on $M,\beta,\nu$ and
\[\alpha = \frac{\beta^2}{\beta+\nu} \in (0,1).\]
\end{lemma}

\noindent
In this connection, it is useful to recall Campanato's classical characterization of H\"older continuous functions. For all $x_0\in\R^N$, $r>0$ we set
\[D_r(x_0) = \Omega\cap B_r(x_0).\]
Given a function $u$ defined in $\Omega$, we define the mean value
\[(u)_{x_0,r} = \fint_{D_r(x_0)}u\,dx\]
and the corresponding $p$-variance
\beq\label{var}
\sigma(u,x_0,r) = \int_{D_r(x_0)}|u-(u)_{x_0,r}|^p\,dx.
\eeq
It is readily checked that
\[\sigma(u,x_0,r) \le \left[\underset{D_r(x_0)}{\rm osc}\,u\right]^p\omega_Nr^N\]
and given $u, v\in L^p(\Omega)$ there exists $C>0$ depending on $N,p$, s.t.\ for any $x_0\in \Omega$, $r>0$
\beq
\label{propvar}
\sigma(u,x_0,r)\le  C\sigma(v,x_0,r)+C\int_{D_r(x_0)}|u-v|^p\,dx.
\eeq
The following result, slightly rephrased from \cite[Theorem 2.9]{G} for our needs, provides a partial converse:

\begin{theorem}\label{cam}
{\rm (Campanato)} Let $\Omega\subset\R^N$ be open, $\mu>0$ s.t.\ for all $x_0\in\Omega$, $0<r<{\rm diam}(\Omega)$
\[|D_r(x_0)| \ge \mu r^N.\]
Also let $p\ge 1$, $0<\alpha<1$, $r_0>0$. Then, there exists $C>0$ depending on $N,\mu,p,\alpha$, and $r_0$, s.t.\ for all $u\in L^p(\Omega)$
\[[u]_{C^\alpha(\overline\Omega)}^p \le C\sup_{x_0\in\Omega,\,0<r<r_0}\,\frac{\sigma(u,x_0,r)}{r^{N+p\alpha}} \ (\le\infty).\]
\end{theorem}

\noindent
Note that Theorem \ref{cam} holds, in particular, if $\Omega$ is bounded and with a $C^{1,1}$-boundary, due to the following Lemma \ref{geo}.

\subsection{Exterior ball condition and Hardy's inequality}\label{ss22}

Aiming at a careful examination of the boundary behaviour of the solution of \eqref{dir}, we obviously need some precise geometric definitions. Recalling that $ \dd_\Omega$ denotes the distance function from $\partial\Omega$,  for all $\rho>0$ we set
\[\Omega_\rho = \big\{x\in\Omega:\, \dd_\Omega(x)<\rho\big\}.\]
We also define the {\em inradius} of $\Omega$ as the supremum of radii of all balls contained in $\Omega$, that is,
\[R_\Omega = \sup_{x\in\Omega}\, \dd_\Omega(x).\]
We say that $\Omega$ satisfies the {\em exterior ball condition} with radius $\rho>0$, shortly ${\rm EBC}(\rho)$, if for all $\bar x\in\partial\Omega$ there exists $y\in\Omega^c$ s.t.\ $B_\rho(y)\subset\Omega^c$ and $\overline{B}_\rho(y)\cap\overline\Omega=\{\bar x\}$ (namely, $B_\rho(y)$ is externally tangent to $\partial\Omega$ at $\bar x$). Clearly ${\rm EBC}(\rho)$ implies ${\rm EBC}(\rho')$ for all $0<\rho'<\rho$.

\begin{remark}\label{ecb}
If $K\subset\R^N$ is open and convex, it is easily seen that $K$ satisfies ${\rm EBC}(\rho)$ for all $\rho>0$. Since $\partial(\Omega\cap K) \subseteq  \partial\Omega \cup \partial K$, if $\Omega$ satisfies ${\rm EBC}(\rho)$ for some $\rho>0$, then so does $\Omega\cap K$. This applies in particular to the set $D_r(x_0)$, defined as in Subsection \ref{ss21} for all $x_0\in\overline\Omega$, $r>0$.
\end{remark}

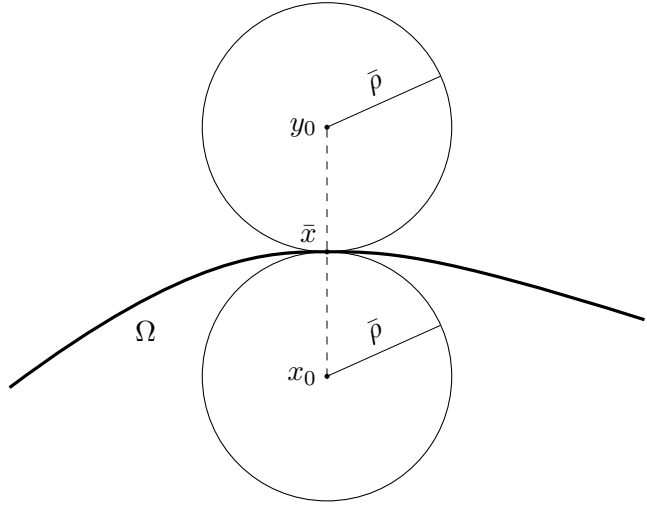
\begin{figure}
\centering
\begin{tikzpicture}[scale=1.5]
\draw[very thin, dashed] (0, -1.1) -- (0, 1.1);
\draw(-1.6, -0.7) node{$\Omega$};
\draw[very thick] (-2.8, -1.2)  .. controls (-1.2, 0) and (-0.5, 0) .. (0, 0) .. controls (0.5, 0) and (0.9, 0) .. (2.8, -0.6);
\draw (0, 1.1) node[left]{$y_0$} circle (1.1cm);
\draw (0, 0) node[above left]{$\bar x$};
\draw (0, -1.1)   node[left]{$x_0$} circle (1.1cm) (0, 1.1);
\filldraw (0, 0) circle (0.5pt) (0, 1.1) circle (0.5pt) (0, -1.1) circle (0.5pt);
\draw[very thin] (0, 1.1) -- node[above, midway, sloped]{$\bar\rho$} (1, 1.55) (0, -1.1) -- node[above, midway, sloped]{$\bar\rho$} (1, 1.55-2.2);
\end{tikzpicture}
\caption{The interior and exterior tangent balls.}
\label{fig1}
\end{figure}

\noindent
The exterior ball condition is naturally connected to the regularity of the boundary. We recall from \cite[Lemma 3.5]{IMS} the following geometrical properties of $C^{1,1}$-smooth domains (see Figure \ref{fig1}):

\begin{lemma}\label{geo}
Let $\Omega\subset\R^N$ be open, bounded, and with a $C^{1,1}$-boundary. Then, there exists $\bar\rho>0$ s.t.\ for all $\bar x\in\partial\Omega$ there are $x_0,y_0\in\R^N$ lying on the normal line to $\partial\Omega$ through $\bar x$ with the following properties:
\begin{enumroman}
\item\label{geo1} $B_{\bar\rho}(x_0)\subset\Omega$, $B_{\bar\rho}(y_0)\subset\Omega^c$;
\item\label{geo2} $\overline{B}_{\bar\rho}(x_0)\cap\overline{B}_{\bar\rho}(y_0)=\{\bar x\}$;
\item\label{geo3} $ \dd_\Omega(x)=|x-\bar x|$ for all $x$ lying on the line segment joining $x_0$ and $\bar x$.
\end{enumroman}
\end{lemma}

\noindent
As a consequence of Lemma \ref{geo} above, any $C^{1,1}$-smooth domain $\Omega$ satisfies ${\rm EBC}(\bar\rho)$ along with $\Omega^c$. In addition, for all $x\in\Omega_{\bar\rho}$ the nearest point on the boundary is {\em unique}.
\vskip2pt
\noindent
The main result of this subsection is a general fractional analogue of Hardy's inequality, under general geometrical conditions. The following discussion is essentially drawn from \cite{BC,D}, but we detail it for the reader's convenience. We begin with a technical lemma (which we shall use independently of Hardy's inequality):

\begin{lemma}\label{sin}
Let $\Omega\subset\R^N$ be an open bounded set satisfying ${\rm EBC}(\rho)$ and with inradius $R_\Omega<\infty$. For any  $\alpha>0$, there exists $C>1$ depending on $N,\alpha$ s.t.\ for all $x_0\in\Omega$
\[\min\Big\{1,\,\frac{\rho^N}{R_\Omega^N}\Big\}\,\frac{C^{-1}}{ \dd_\Omega^\alpha(x_0)} \le \int_{\Omega^c}\frac{dx}{|x-x_0|^{N+\alpha}} \le \frac{C}{ \dd_\Omega^\alpha(x_0)}.\]
\end{lemma}
\begin{proof}
Let $ \dd_\Omega(x_0)=r>0$. Also, let $\bar x\in\partial\Omega$ be s.t.\ $|x_0-\bar x|=r$, and $B_\rho(y_0)$ be exteriorly tangent to $\partial\Omega$ at $\bar x$. To prove the left hand side inequality, we note that for all $x\in B_\rho(y_0)$ 
\[|x-x_0| \le |x-y_0|+|y_0-\bar x|+|\bar x-x_0| < 2\rho+r,\]
hence we have
\beq\label{sin1}
\int_{\Omega^c}\frac{dx}{|x-x_0|^{N+\alpha}} \ge \int_{B_\rho(y_0)}\frac{dx}{|x-x_0|^{N+\alpha}} \ge \frac{|B_\rho(y_0)|}{(2\rho+r)^{N+\alpha}}.
\eeq
Now we distinguish two cases. If $\rho\le r$, then 
\[\frac{|B_\rho(y_0)|}{(2\rho+r)^{N+\alpha}}  \ge \frac{\omega_N\rho^N}{(3r)^{N+\alpha}}\ge \frac{\omega_N}{3^{N+\alpha} }\frac{\rho^N}{R^N_\Omega} \frac{1}{r^\alpha}.\]
If $\rho>r$, then $\Omega$ satisfies ${\rm EBC}(r)$ as well, so can we assume $\rho=r$ in the previous computation to get
\[\frac{|B_\rho(y_0)|}{(2\rho+r)^{N+\alpha}}=\frac{\omega_N}{3^{N+\alpha}} \frac{1}{r^\alpha}. \]
Plugging the estimates above into \eqref{sin1}, we find in either case
\[\int_{\Omega^c}\frac{dx}{|x-x_0|^{N+\alpha}} \ge \min\Big\{1,\,\frac{\rho^N}{R_\Omega^N}\Big\}\,\frac{C^{-1}}{r^\alpha}.\]
For the right hand side inequality, we first note that $|x-x_0|\ge r$ for all $x\in\Omega^c$, $|x_0-y_0|<2r$, and for all $x\in\Omega^c\cap B_r^c(y_0)$
\[|x-x_0| \ge |x-y_0|-|y_0-x_0| > |x-y_0|-2|x-x_0|,\]
which implies $|x-x_0|>|x-y_0|/3$. Then we split $\Omega^c$ into two parts:
\begin{align*}
\int_{\Omega^c}\frac{dx}{|x-x_0|^{N+\alpha}} &= \int_{\Omega^c\cap B_r(y_0)}\frac{dx}{|x-x_0|^{N+\alpha}}+\int_{\Omega^c\cap B_r^c(y_0)}\frac{dx}{|x-x_0|^{N+\alpha}} \\
&\le \frac{|B_r(y_0)|}{r^{N+\alpha}}+3^{N+\alpha}\int_{B_r^c(y_0)}\frac{dx}{|x-y_0|^{N+\alpha}} \le \frac{C}{r^\alpha}.
\end{align*}
Thus, both bounds are proved.
\end{proof}

\noindent
The lower bound in Lemma \ref{sin} may vanish for bounded $\Omega$ without the exterior ball condition, i.e.\, with interior angles. It can also vanish for regular domains with $R_\Omega=\infty$, for example when $\Omega=\overline{B}_1^c(0)$, (which satisfies the exterior ball condition with $\rho=1$). We turn now to the following version of Hardy's inequality for the fractional Sobolev space $\w$, corresponding to \cite[Theorem 1.1]{D}:

\begin{theorem}\label{fhi}
{\rm (Fractional Hardy's inequality)} Let $\Omega\subset\R^N$ be open with $R_\Omega<\infty$, satisfying ${\rm EBC}(\rho)$ with $\rho>0$, $p>1$, $0<s<1$. Then, there exists $C>0$ depending on $N,p,s$, s.t.\ for all $u\in\w$
\[\Big\|\frac{u}{\ds}\Big\|_{L^p(\Omega)}^p \le C\max\Big\{1,\,\frac{R^N_\Omega}{\rho^N}\Big\}[u]_{s,p}^p.\]
\end{theorem}
\begin{proof}
Let $u\in\w$. We start from the Gagliardo norm of $u$, reduce the integral using symmetry and recalling that $u=0$ in $\Omega^c$, then we apply Lemma \ref{sin} with $\alpha=ps$ to the integral in $y$:
\begin{align*}
[u]_{s,p}^p &= \iint_{\R^N\times\R^N}\frac{|u(x)-u(y)|^p}{|x-y|^{N+ps}}\,dx\,dy \\
&\ge 2\iint_{\Omega\times\Omega^c}\frac{|u(x)|^p}{|x-y|^{N+ps}}\,dx\,dy \\
&\ge \frac{2}{C}\min\Big\{1,\,\frac{\rho^N}{R^N_\Omega}\Big\}\int_\Omega\frac{|u|^p}{\ds}\,dx.
\end{align*}
Reversing the constant, the desired inequality is achieved.
\end{proof}

\begin{remark}\label{con}
(Convex domains) The case of convex domains is extensively treated in \cite{BC}, where a fractional Hardy's inequality is proved with a constant that exhibits optimal asymptotics as $s\to 0,1$, respectively. In our setting, we can observe that any convex $\Omega$ satisfies ${\rm EBC}(\rho)$ with arbitrarily large $\rho$, hence for all $x_0\in\Omega$ we may choose $\rho> \dd_\Omega(x_0)$, and the estimate of Lemma \ref{sin} then becomes the easier following one:
\[\frac{C^{-1}}{ \dd_\Omega^\alpha(x_0)} \le \int_{\Omega^c}\frac{dx}{|x-x_0|^{N+\alpha}} \le \frac{C}{ \dd_\Omega^\alpha(x_0)}.\]
Therefore, if $\Omega$ is convex, we can drop the finite inradius assumption in Theorem \ref{fhi}.
\end{remark}

\subsection{Monotonicity}\label{ss23}

Among the properties of $\fpl$, monotonicity plays a fundamental role, but it is affected by the nonlocal nature of the operator.
\vskip2pt
\noindent
In the degenerate regime, we have a quite natural inequality:

\begin{lemma}\label{mod}
{\rm (Monotonicity, degenerate case)} Let $p\ge 2$, $u,v\in\w$. Then, there exists $C>0$ depending on $p$ s.t.\
\[[u-v]_{s,p}^p \le C\langle\fpl u-\fpl v,u-v\rangle.\]
\end{lemma}
\begin{proof}
We recall from \cite[Lemma A.2]{BP} the following inequality: for all $a,b,c,d\in\R$ with $a-b=c-d$ we have
\[|a-b|^p \le C(a-b)(c^{p-1}-d^{p-1}),\]
with $C>0$ depending on $p$. Choosing $a=u(x)-v(x)$, $b=u(y)-v(y)$, and accordingly $c=u(x)-u(y)$, $d=v(x)-v(y)$, and integrating we find
\begin{align*}
[u-v]_{s,p}^p &= \iint_{\R^N\times\R^N}\frac{|(u(x)-v(x))-(u(y)-v(y))|^p}{|x-y|^{N+ps}}\,dx\,dy \\
&\le C\iint_{\R^N\times\R^N}\frac{(u(x)-u(y))^{p-1}-(v(x)-v(y))^{p-1}}{|x-y|^{N+ps}}\left[(u(x)-v(x))-(u(y)-v(y))\right]\,dx\,dy \\
&= C\langle\fpl u-\fpl v,u-v\rangle,
\end{align*}
with $C>0$ depending on $p$ alone.
\end{proof}

\noindent
In the singular regime, a more involved inequality holds under appropriate geometrical conditions. In addition, we introduce a weighted monotonicity formula that will serve our endings:

\begin{lemma}\label{mos}
{\rm (Monotonicity, singular case)} Let $1<p<2$, $\Omega$ have inradius $R_\Omega<\infty$ and satisfy ${\rm EBC}(\rho)$. Then, there exists $C>0$ depending on $N,p,s$ s.t.\ setting
\[C_\Omega = C\max\Big\{1,\,\frac{R_\Omega^N}{\rho^N}\Big\},\]
for all $u,v\in\w$ and all bounded $\Omega'\subseteq\Omega$ the following inequalities hold:
\begin{enumroman}
\item\label{mos1} $\displaystyle\|u-v\|_{L^p(\Omega')}^p \le \left[C_\Omega\| \dd_\Omega\|_{L^\infty(\Omega')}^{ps}\langle\fpl u-\fpl v,u-v\rangle\right]^\frac{p}{2}\Big[\|u\|_{L^p(\Omega')}^p+\|v\|_{L^p(\Omega')}^p\Big]^\frac{2-p}{2};$
\item\label{mos2} $\displaystyle\Big\|\frac{u-v}{\ds}\Big\|_{L^p(\Omega')}^p \le \big[C_\Omega\langle\fpl u-\fpl v,u-v\rangle\big]^\frac{p}{2}\Big[\Big\|\frac{u}{\ds}\Big\|_{L^p(\Omega')}^p+\Big\|\frac{v}{\ds}\Big\|_{L^p(\Omega')}^p\Big]^\frac{2-p}{2}.$
\end{enumroman}
\end{lemma}
\begin{proof}
We recall from \cite[Lemma A.2]{BP} the following inequality: when $p\in (1, 2)$, for all $a,b\in\R$
\[(a-b)^2 \le C(a^{p-1}-b^{p-1})(a-b)(a^2+b^2)^\frac{2-p}{p},\]
with $C>0$ depending on $p$. We apply such inequality with $a=u(x)$, $b=v(x)$, raise both sides to the power $p/2$, and integrate in $\Omega'$. Then we apply H\"older's inequality introducing the factor $ \dd_\Omega^{ps}$:
\begin{align}\label{mos3}
&\int_{\Omega'}|u-v|^p\,dx \le C\int_{\Omega'}(u^{p-1}-v^{p-1})^\frac{p}{2}(u-v)^\frac{p}{2}(|u|^p+|v|^p)^\frac{2-p}{2}\,dx \\
\nonumber &\le C\| \dd_\Omega\|_{L^\infty(\Omega')}^\frac{p^2s}{2}\left[\int_{\Omega'}\frac{(u^{p-1}-v^{p-1})(u-v)}{ \dd_\Omega^{ps}}\,dx\right]^\frac{p}{2}\left[\|u\|_{L^p(\Omega')}^p+\|v\|_{L^p(\Omega')}^p\right]^\frac{2-p}{2}.
\end{align}
By Lemma \ref{sin} (with $\alpha=ps$), for all $x\in\Omega'\subseteq \Omega$ we have
\[\frac{1}{ \dd_\Omega^{ps}(x)} \le C_\Omega\int_{\Omega^c}\frac{dy}{|x-y|^{N+ps}},\]
with $C>0$ depending on $N,p,s,$ and $\Omega$. Hence we can estimate the first integral above (which has a non-negative integrand)  as follows:
\begin{align*}
\int_{\Omega'}\frac{(u^{p-1}(x)-v^{p-1}(x))(u(x)-v(x))}{ \dd_\Omega^{ps}(x)}\,dx &\le C_\Omega\iint_{\Omega'\times\Omega^c}\frac{(u^{p-1}(x)-v^{p-1}(x))(u(x)-v(x))}{|x-y|^{N+ps}}\,dx\,dy \\
&\le C_\Omega\langle\fpl u-\fpl v,u-v\rangle
\end{align*}
where in the last inequality we used the fact that $u(y)=v(y)=0$ for $y\in \Omega^c$.
Plugging such inequality into \eqref{mos3}, we get \ref{mos1}. \\
Inequality \ref{mos2} is in fact easier to prove. Without loss of generality, we assume that $u/\ds,v/\ds\in L^p(\Omega')$. Starting from the pointwise inequality, we have as above
\[\int_{\Omega'}\frac{|u-v|^p}{ \dd_\Omega^{ps}}\, dx \le C\left[\int_{\Omega'}\frac{(u^{p-1}-v^{p-1})(u-v)}{ \dd_\Omega^{ps}} \, dx\right]^\frac{p}{2}\left[\Big\|\frac{u}{\ds}\Big\|_{L^p(\Omega')}^p+\Big\|\frac{u}{\ds}\Big\|_{L^p(\Omega')}^p\right]^\frac{2-p}{2},\]
where H\"older's inequality is applied this time with the measure $dx/ \dd_\Omega^{ps}(x)$. The rest of the argument runs as in the previous case, leading to \ref{mos2}.
\end{proof}

\section{Harmonic extensions} \label{sec3}

\noindent
Our approach is partly based on a classical perturbative method, as in previous studies on fractional regularity theory (see \cite{BLS,GL}), hence we need some fine estimates on solutions of homogeneous equations in special domains. We say that a function $v\in\widetilde{W}^{s,p}(\Omega)$ is $(s, p)$-{\em harmonic} in $\Omega$ if it satisfies
\[\fpl v = 0 \qquad \text{in $\Omega$,}\]
in the sense of equation \eqref{fee}. 
Given a function $u\in\w$, $x_0\in\overline\Omega$, and $R>0$, we define the $(s, p)$-{\em harmonic extension} of $u$ in $D_R(x_0)\subseteq\Omega$ as the unique solution $v\in\w$ of the problem
\beq\label{hex}
\begin{cases}
\fpl v = 0 & \text{in $D_R(x_0)$} \\
v = u & \text{in $D_R^c(x_0)$.}
\end{cases}
\eeq
Note that both $u,v$ vanish in $\Omega^c$.
In all the following results, $\Omega$ is assumed to be bounded and with a $C^{1,1}$-smooth boundary.

\noindent
A basic tool for our purposes will be the following barrier, constructed in  \cite[Lemma 4.3]{IMS}.

\begin{proposition}
\label{bar}
There exist $0<\eps<1$, $C>1$ depending on $N,p,s$ with the following property. Given $e\in\partial B_1(0)$, there exists a function $w\in C^s(\R^N)\cap \widetilde{W}^{s, p}(B_\eps(e)\setminus\overline{B}_1(0))$ s.t.\
\[\begin{cases}
\displaystyle\fpl w \ge \frac{1}{C} & \text{in $B_\eps(e)\setminus\overline{B}_1(0)$} \\
\displaystyle\frac{\dd_1^s}{C} \le w \le C\dd_1^s & \text{in $\R^N$,}
\end{cases}\]
where $\dd_1={\rm dist}(\cdot,B_1(0))$.
\end{proposition}

\subsection{Oscillation estimates}\label{ss31}

From now on, we focus on the setting where $v$ is the $(s, p)$-harmonic extension of $u$ in suitable sets, deriving estimates for $v$ in terms of $u$.  We first consider $L^\infty$ bounds:
 
\begin{lemma}\label{abh}
Let $u,v\in\w$ with $v$ being the $(s,p)$-harmonic extension of $u$ in $D_R(x_0)\subseteq\Omega$, $0\le\alpha<\min\{1, p's\}$, and $H_\alpha(u)$ be defined by \eqref{hal}. Then, there exists $C>0$ depending on the data and $\alpha$, s.t.\
\[\|v\|_{L^\infty(D_R(x_0))} \le C\|u\|_{L^\infty(D_{2R}(x_0))}+CH_\alpha(u)R^\alpha.\]
\end{lemma}
\begin{proof}
First we consider the case $\alpha=0$ and assume, without loss of generality, that $u\in L^\infty(\Omega)$. Then, recalling \eqref{hex}, we may apply Lemma \ref{wcp} to $v$ and the constant $\|u\|_{L^\infty(\Omega)}$ (which is obviously harmonic in $D_R(x_0)$) to find that in $D_R(x_0)$
\[v \le \|u\|_{L^\infty(\Omega)} \le \|u\|_{L^\infty(D_{2R}(x_0))}+H_0(u).\]
Now we turn to the case $\alpha>0$ and assume $u\in C^\alpha(\overline\Omega)$. In particular, then, we may set
\[M_R = \|u\|_{L^\infty(D_{2R}(x_0))} < \infty.\]
Set for all $x\in\R^N$
\[\tilde v(x) = \begin{cases}
v(x) & \text{if $x\in D_R(x_0)$} \\
\min\{u(x),\,M_R\} & \text{if $x\in D_R^c(x_0)$.}
\end{cases}\]
We have $\tilde v=u$ in $D_{2R}(x_0)$, hence $\tilde v-v$ is supported away from $D_R(x_0)$. By Lemma \ref{spp}, we have for all $x\in D_R(x_0)$
\begin{align*}
\fpl\tilde v(x) &= \fpl v(x)+2\int_{D_R^c(x_0)\cap\{u>M_R\}}\frac{(v(x)-\tilde v(y))^{p-1}-(v(x)-v(y))^{p-1}}{|x-y|^{N+ps}}\,dy \\
&= 2\int_{D_R^c(x_0)\cap\{u>M_R\}}\frac{(u(y)-v(x))^{p-1}-(M_R-v(x))^{p-1}}{|x-y|^{N+ps}}\,dy =: h(x).
\end{align*}
Since $D_R^c(x_0)\cap\{u>M_R\}\subseteq D_{2R}^c(x_0)$ and the mapping $t\mapsto t^{p-1}$ is increasing in $\R$, we have $h\ge 0$ in $D_R(x_0)$. Besides, for all $x\in D_R(x_0)$, $y\in D_{2R}^c(x_0)$ we have
\[|x-y| \ge |x_0-y|-|x-x_0| \ge \frac{|x_0-y|}{2}.\]
Therefore,
\beq\label{abh1}
0 \le h(x) \le C\int_{D_{2R}^c(x_0)}\frac{(u(y)-v(x))^{p-1}-(M_R-v(x))^{p-1}}{|x_0-y|^{N+ps}}\,dy.
\eeq
The estimate above is improved in different ways in the degenerate and singular cases, respectively:
\begin{itemize}[leftmargin=1cm]
\item[$(a)$] If $p\ge 2$, then we use in \eqref{abh1} the following point-wise inequality holding for all $a\ge b$, $c\in\R$ (which follows at once from \cite[formula (2.15)]{IMS1}):
\[(a-c)^{p-1}-(b-c)^{p-1} \le C(a-b)[(a-b)^{p-2}+|b|^{p-2}+|c|^{p-2}].\]
Hence, we have for all $x\in D_R(x_0)$
\begin{align*}
h(x) &\le C\int_{D_{2R}^c(x_0)}\frac{|u(y)-M_R|^{p-1}}{|x_0-y|^{N+ps}}\,dy+C\big(M_R^{p-2} \\
&+ \|v\|_{L^\infty(D_R(x_0))}^{p-2}\big)\int_{D_{2R}^c(x_0)}\frac{|u(y)-M_R|}{|x_0-y|^{N+ps}}\,dy.
\end{align*}
Since $u\in C^\alpha(\overline\Omega)$, we can find $x_1\in\overline{D}_{2R}(x_0)$ s.t.\ $|u(x_1)|=M_R$. So, for all $y\in D_{2R}^c(x_0)$ we have
\[|y-x_1| \le |y-x_0|+|x_0-x_1| \le 2|x_0-y|,\]
which in turn implies
\[
|u(y)-M_R| = |u(y)-u(x_1)| \le [u]_{C^\alpha(\overline\Omega)}|y-x_1|^\alpha \le 2^\alpha[u]_{C^\alpha(\overline\Omega)}|x_0-y|^\alpha.
\]
Recalling that $\alpha(p-1)<ps$, we can estimate the first integral above as follows:
\begin{align*}
\int_{D_{2R}^c(x_0)}\frac{|u(y)-M_R|^{p-1}}{|x_0-y|^{N+ps}}\,dy &\le C[u]_{C^\alpha(\overline\Omega)}^{p-1}\int_{D_{2R}^c(x_0)}\frac{dy}{|x_0-y|^{N+ps-(p-1)\alpha}} \le \frac{C[u]_{C^\alpha(\overline\Omega)}^{p-1}}{R^{ps-(p-1)\alpha}}.
\end{align*}
Similarly, noting that also $\alpha<ps$, we have for the second integral
\[\int_{D_{2R}^c(x_0)}\frac{|u(y)-M_R|}{|x_0-y|^{N+ps}}\,dy \le \frac{C[u]_{C^\alpha(\overline\Omega)}}{R^{ps-\alpha}}.\]
Plugging these inequalities into the previous one, we find $C>0$ depending on the data and $\alpha$, s.t.\ for all $x\in D_R(x_0)$
\beq\label{abh2}
h(x) \le \frac{C[u]_{C^\alpha(\overline\Omega)}^{p-1}}{R^{ps-(p-1)\alpha}}+\frac{C[u]_{C^\alpha(\overline\Omega)}}{R^{ps-\alpha}}\big(M_R^{p-2}+\|v\|_{L^\infty(D_R(x_0))}^{p-2}\big).
\eeq
\item[$(b)$] If $p<2$, then we use the following simpler point-wise inequality (see \cite[formula (A.5)]{IM}), holding for all $a\ge b$, $c\in\R$:
\[(a-c)^{p-1}-(b-c)^{p-1} \le (a-b)^{p-1}.\]
Starting from \eqref{abh1} and arguing as above, we find $C>0$ depending on the data and $\alpha$ s.t.\ for all $x\in D_R(x_0)$
\beq\label{abh3}
h(x) \le C\int_{D_R^c(x_0)\cap\{u>M_R\}}\frac{(u(y)-M_R)^{p-1}}{|x_0-y|^{N+ps}}\,dy \le \frac{C[u]_{C^\alpha(\overline\Omega)}^{p-1}}{R^{ps-(p-1)\alpha}},
\eeq
(note that we only used $(p-1)\alpha<ps$).
\end{itemize}
Now define $w\in W^{s,p}_0(B_{2R}(x_0))$ as the unique solution of the torsion problem
\[\begin{cases}
\fpl w = 1 & \text{in $B_{2R}(x_0)$} \\
w = 0 & \text{in $B_{2R}^c(x_0)$.}
\end{cases}\]
By \cite[Lemma 2.2]{IMS1}, there is $C>0$ depending on $N,p,s$ s.t.\ in $\R^N$ there holds $0 \le w \le CR^{p's}$. Set for all $x\in\R^N$
\[\tilde w(x) = M_R+\|h\|_{L^\infty(D_R(x_0))}^\frac{1}{p-1}w(x).\]
Then $\tilde w$ is an upper barrier for $\tilde v$ in the following sense:
\[\begin{cases}
\fpl\tilde v \le \|h\|_{L^\infty(D_R(x_0))} = \fpl\tilde w & \text{in $D_R(x_0)$} \\
\tilde v \le M_R \le \tilde w & \text{in $D_R^c(x_0)$.}
\end{cases}\]
By Lemma \ref{wcp} we have $\tilde v\le\tilde w$ in all of $R^N$. In particular, for all $x\in D_R(x_0)$
\begin{align*}
v(x) &\le M_R+\|h\|_{L^\infty(D_R(x_0))}^\frac{1}{p-1}\|w\|_{L^\infty(B_{2R}(x_0))} \\
&\le M_R+C\|h\|_{L^\infty(D_R(x_0))}^\frac{1}{p-1}R^{p's}.
\end{align*}
Again we must distinguish two cases. If $p\ge 2$, we start from the previous estimate of $v$, apply \eqref{abh2}, and conclude via Young's inequality:
\begin{align*}
\sup_{D_R(x_0)}\,v &\le M_R+C\Big[\frac{[u]_{C^\alpha(\overline\Omega)}^{p-1}}{R^{ps-(p-1)\alpha}}+\frac{[u]_{C^\alpha(\overline\Omega)}}{R^{ps-\alpha}}\big(M_R^{p-2}+\|v\|_{L^\infty(D_R(x_0))}^{p-2}\big)\Big]^\frac{1}{p-1}R^{p's} \\
&\le M_R+C[u]_{C^\alpha(\overline\Omega)}R^\alpha+C[u]_{C^\alpha(\overline\Omega)}^\frac{1}{p-1}\big(M_R^\frac{p-2}{p-1}+\|v\|_{L^\infty(D_R(x_0))}^\frac{p-2}{p-1}\big)R^\frac{\alpha}{p-1} \\
&\le (1+\eps)M_R+\eps\|v\|_{L^\infty(D_R(x_0))}+C_\eps[u]_{C^\alpha(\overline\Omega)}R^\alpha,
\end{align*}
with $C_\eps\to\infty$ as $\eps\to 0$. Using the same argument on $-v$, we easily obtain
\[\|v\|_{L^\infty(D_R(x_0))} \le (1+\eps)M_R+\eps\|v\|_{L^\infty(D_R(x_0))}+C_\eps[u]_{C^\alpha(\overline\Omega)}R^\alpha.\]
Choosing $\eps>0$ small enough, we can reabsorb the $L^\infty$-norm of $v$ on the right hand side into the left hand side. Recalling the definition of $M_R$, we have
\[\|v\|_{L^\infty(D_R(x_0))} \le C\|u\|_{L^\infty(D_{2R}(x_0))}+CH_\alpha(u)R^\alpha.\]
If $p<2$, an easier argument exploiting \eqref{abh3} leads to the same conclusion.
\end{proof}

\noindent
The next lemma yields an oscillation estimate on $v$ near the boundary of $\Omega$:

\begin{lemma}\label{hoe}
Let $\bar x\in\partial\Omega$, $0<R<{\rm diam}(\Omega)$, $u,v\in\w$ with being the $(s, p)$-harmonic extension of $u$ in $D_R(\bar x)$. Also, let $0\le\alpha<\min\{1,p's\}$ and $H_\alpha(u)$ be defined by \eqref{hal}. Then, there exists $C>0$ depending on the data (but not on $\alpha$), s.t.\ for all $0<r\le R$
\[\underset{D_r(\bar x)}{\rm osc}\,v \le CH_\alpha(u) R^\alpha\Big(\frac{r}{R}\Big)^s.\]
\end{lemma}
\begin{proof}
First we point out that
\beq\label{hoe1}
\|u\|_{L^\infty(D_{2R}(\bar x))} \le H_\alpha(u)(2R)^\alpha.
\eeq
Indeed, if $\alpha=0$, then \eqref{hoe1} directly follows from \eqref{hal}. If $\alpha>0$, then assuming $u\in C^\alpha(\overline\Omega)$ (which will be tacitly assumed henceforth), since $u(\bar x)=0$ we have for all $x\in D_{2R}(\bar x)$
\[|u(x)| \le [u]_{C^\alpha(\overline\Omega)}|x-\bar x|^\alpha \le H_\alpha(u)(2R)^\alpha.\]
By Lemma \ref{abh} (with $x_0=\bar x$) and \eqref{hoe1} we have
\beq\label{hoe2}
\|v\|_{L^\infty(D_{2R}(\bar x))} \le C\|u\|_{L^\infty(D_{2R}(\bar x))}+C[u]_{C^\alpha(\overline\Omega)}R^\alpha \le C_1H_\alpha(u) R^\alpha,
\eeq
with $C_1>0$ depending on the data (recall that $R$ is bounded from above and $\alpha\le 1$).  Now we argue as in Lemma \ref{abh}, setting for all $x\in\R^N$
\[\tilde v(x) = \begin{cases}
v(x) & \text{if $x\in D_R(\bar x)$} \\
\min\left\{u(x),\,\|u\|_{L^\infty(D_{2R}(\bar x))}\right\} & \text{if $x\in D_R^c(\bar x)$.}
\end{cases}\]
Up to taking $C_1$ even bigger, we have
\beq\label{hoe3}
\|\tilde v\|_{L^\infty(\Omega)} \le C_1H_\alpha(u) R^\alpha.
\eeq
As in the proof of Lemma \ref{abh} we have
\[\fpl\tilde v = h(x) \ \text{in $D_R(\bar x)$,}\]
with $h$ obeying
\begin{equation}\label{hoe4}
\|h\|_{L^\infty(D_R(\bar x))} \le \frac{CH^{p-1}_\alpha(u)}{R^{ps-(p-1)\alpha}}+\frac{CH_\alpha(u)\|v\|_{L^\infty(D_R(\bar x))}^{p-2}}{R^{ps-\alpha}} \le \frac{C_2H^{p-1}_\alpha(u)}{R^{ps-(p-1)\alpha}},
\end{equation}
where we have absorbed the second term into the first one by using \eqref{hoe2}. Next, let $0<\bar\rho<1$ be as in Lemma \ref{geo}, and assume $R\le\bar\rho/2$. Then we can find $y_0,y_1\in\Omega^c$ s.t.\ both balls $B_R(y_0)$, $B_{2R}(y_1)$ are exteriorly tangent to $\partial\Omega$ at $\bar x$ (see Figure \ref{fig2}).

\begin{figure}
\centering
\begin{tikzpicture}[scale=2.5]
\clip (-2.2, -1) rectangle (2.2, 1.6);
\begin{scope}
\clip (-2.2, -1)  .. controls (-1.2, 0) and (-0.5, 0) .. (0, 0) .. controls (0.5, 0) and (0.9, 0) .. (2.2, -0.6);
\filldraw[lightgray] (0, 0) circle (0.55);
\draw (0, 0) circle (0.55);
\end{scope} 
\draw(-1.6, -0.7) node{$\Omega$};
\draw[very thick] (-2.2, -1)  .. controls (-1.2, 0) and (-0.5, 0) .. (0, 0) .. controls (0.5, 0) and (0.9, 0) .. (2.2, -0.6);
\draw (0, 1.1) node[above]{$y_1$} circle (1.1cm);
\draw (0, 0.55) node[left]{$y_0$} circle (0.55cm);
\draw (0, 0) node[below]{$\bar x$};
\draw (0, 0)  circle (0.3cm);
\filldraw (0, 0) circle (0.5pt) (0, 1.1) circle (0.5pt) (0, 0.55) circle (0.5pt);
\draw[very thin] (0, 1.1) -- node[above, midway, sloped]{$2R$} (1, 1.55) (0, 0.55) -- node[above, midway, sloped]{$R$} (0.5,0.78) (0, 0) -- node[above, midway, sloped]{$\eps R$} (0.28,0.11) ;
\end{tikzpicture}
\caption{The set $D_R(\bar x)$ in grey and the other auxiliary points.}
\label{fig2}
\end{figure}
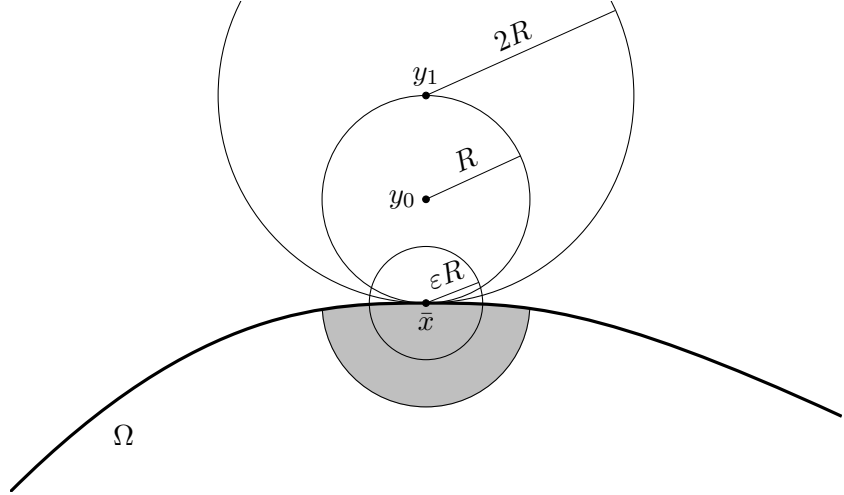

\vskip2pt
\noindent
Let now $\eps\in(0,1)$, $w$ be given by Proposition \ref{bar} for $e=(\bar x-y_0)/R$ and set 
\[w_R(x) = w\Big(\frac{x-y_0}{R}\Big).\]
Then, still we have $w_R\in C^s(\R^N)\cap \widetilde{W}^{s, p}(B_{\eps R}(\bar x)\setminus\overline{B}_R(y_0))$ and by the homogeneity and scaling properties of $\fpl$  we have
\beq\label{hoe5}
\begin{cases}
\displaystyle\fpl w_R \ge \frac{1}{C_3R^{ps}} & \text{in $B_{\eps R}(\bar x)\setminus\overline{B}_R(y_0)$} \\
\displaystyle\frac{  \dd_R^s}{C_3} \le w_R \le C_3  \dd_R^s & \text{in $\R^N$,}
\end{cases}
\eeq
where $C_3>0$ depends on $N,p,s$ and for all $x\in\R^N$ we have set for simplicity
\[  \dd_R(x) = \frac{1}{R}{\rm dist}(x,B_R(y_0)).\]
Fix $0<r\le R$, and distinguish two cases:
\begin{itemize}[leftmargin=1cm]
\item[$(a)$] If $r\ge\eps R$, then by \eqref{hoe2} we have
\[\underset{D_r(\bar x)}{\rm osc}\,v \le 2\|v\|_{L^\infty(D_R(\bar x))} \le 2C_1H_\alpha(u) R^\alpha \le \frac{2C_1}{\eps^s}H_\alpha(u) R^\alpha\Big(\frac{r}{R}\Big)^s,\]
which is the conclusion with $C=2C_1/\eps^s>0$ depending on the data.
\item[$(b)$] If $r<\eps R$, then $D_r(\bar x)\subset B_{\eps R}(\bar x)\setminus\overline{B}_R(y_0)$ and $\Omega\setminus\overline{D}_{\eps R}(\bar x)\subset\overline{B}_{2R}^c(y_1)\setminus\overline{B}_{\eps R}(\bar x)$. By such inclusions we have
\[
\inf_{\Omega\setminus\overline{D}_{\eps R}(\bar x)}\,  \dd_R \ge \inf_{\overline{B}_{2R}^c(y_1)\setminus\overline{B}_{\eps R}(\bar x)}\,  \dd_R \\
= \sqrt{\frac{\eps^2}{2}+1}-1 =: c,
\]
with $c\in(0,1)$ depending on the data, and the final passage deriving from a direct geometrical computation. Set
\[\mu = \max\Big\{(C_2C_3)^\frac{1}{p-1},\,\frac{C_1C_3}{c^s}\Big\}H_\alpha(u),\]
with $C_i$ ($i=1,2,3$) as in the previous estimates, so that $\mu>0$ depends on $N,p,s$, and on $\alpha$ through $H_\alpha(u)$ only. Set also for all $x\in\R^N$
\[\tilde w(x) = \mu R^\alpha w_R(x).\]
We will now see that $\tilde w$ is a barrier for $\tilde v$ in $D_{\eps R}(\bar x)$. By \eqref{hoe5}, \eqref{hoe4}, and the definition of $\mu$, we have for all $x\in D_{\eps R}(\bar x)$
\begin{align*}
\fpl\tilde w(x) &= \mu^{p-1}R^{(p-1)\alpha}\fpl w_R(x) \\
&\ge \frac{\mu^{p-1}}{C_3R^{ps-(p-1)\alpha}} \ge \frac{C_2H_\alpha^{p-1}(u)}{R^{ps-(p-1)\alpha}} \ge h(x).
\end{align*}
Besides, by \eqref{hoe5}, the lower bound on $  \dd_R$, the definition of $\mu$, and \eqref{hoe3}, for all $x\in\Omega\setminus\overline{D}_{\eps R}(\bar x)$ we have
\begin{align*}
\tilde w(x) &\ge \frac{\mu R^\alpha  \dd_R^s(x)}{C_3} \ge \frac{C_1H_\alpha(u)\dd_R^s(x)}{c^s} \\
&\ge C_1H_\alpha(u) R^\alpha \ge \tilde v(x),
\end{align*}
while obviously for all $x\in\Omega^c$
\[\tilde w(x) \ge 0 = \tilde v(x).\]
Summarizing, we have
\[\begin{cases}
\fpl\tilde v \le \fpl\tilde w & \text{in $D_{\eps R}(\bar x)$} \\
\tilde v \le \tilde w & \text{in $D_{\eps R}^c(\bar x)$.}
\end{cases}\]
By Lemma \ref{wcp} we deduce $\tilde v\le\tilde w$ in all of $\R^N$. In particular, by \eqref{hoe5}, for all $x\in D_r(\bar x)\subset D_{\eps R}(\bar x)$ we have
\[\tilde v(x) \le \tilde w(x) \le C_3\mu R^\alpha  \dd_R^s(x).\]
Arguing similarly on $-\tilde v$, we find the symmetrical estimate
\[\tilde v(x) \ge -C_3\mu R^\alpha  \dd_R^s(x).\]
By construction, for all $x\in D_r(\bar x)$ we have both $\tilde v(x)=v(x)$ and
\[\dd_R(x) \le \frac{|x-\bar x|}{R} \le \frac{r}{R}.\]
Therefore, by definition of $\mu$ and the previous two-sided bounds on $\tilde v$, we have
\begin{align*}
\underset{D_r(\bar x)}{\rm osc}\,v &\le 2C_3\mu R^\alpha\sup_{D_r(\bar x)}\,\dd_R^s \\
&\le 2C_3\max\Big\{(C_2C_3)^\frac{1}{p-1},\,\frac{C_1C_3}{c^s}\Big\}H_\alpha(u) R^\alpha\Big(\frac{r}{R}\Big)^s,
\end{align*}
which yields the conclusion for a $C>0$ depending on the data.
\end{itemize}
There remains to examine the case $R>\bar\rho/2$. Given $0<r\le R$, if $r\le\bar\rho/2$ then we repeat the previous argument with $R$ replaced by $\bar\rho/2$ and find
\begin{align*}
\underset{D_r(\bar x)}{\rm osc}\,v &\le CH_\alpha(u)\Big(\frac{\bar\rho}{2}\Big)^\alpha\Big(\frac{2r}{\bar\rho}\Big)^s \\
&\le C\Big(\frac{2}{\bar\rho}\Big)^s{\rm diam}(\Omega)^sH_\alpha(u) R^\alpha\Big(\frac{r}{R}\Big)^s = C_\Omega H_\alpha(u) R^\alpha\Big(\frac{r}{R}\Big)^s,
\end{align*}
with $C_\Omega>0$ depending on the data (we emphasize the dependance on $\Omega$ in this case). Finally, if $\bar\rho/2<r\le R$, then simply by \eqref{hoe2}
\begin{align*}
\underset{D_r(\bar x)}{\rm osc}\,v &\le 2\|v\|_{L^\infty(D_R(\bar x))} \le 2C_1H_\alpha(u) R^\alpha \\
&\le CH_\alpha(u) R^\alpha\Big(\frac{2r}{\bar\rho}\Big)^s \le C_\Omega H_\alpha(u) R^\alpha\Big(\frac{r}{R}\Big)^s,
\end{align*}
again with $C_\Omega>0$ depending on the data. Thus, the desired estimate is achieved in all cases.
\end{proof}

\subsection{Weighted estimates}\label{ss32}

We next focus on estimates for $v/\ds$, when $v$ is the $(s, p)$-harmonic extension of a given $u\in \w$.
The following weighted oscillation estimate near the boundary can be obtained by slightly adapting the proofs of \cite[Theorem 5.1]{IMS1} (for the degenerate regime) and \cite[Proposition 5.1]{IM} (for the singular regime), respectively:

\begin{proposition}\label{boe} 
Let $\bar x\in\partial\Omega$, $0<R\le\bar\rho$ (with $\bar\rho$ as in Lemma \ref{geo}), and $v\in\w$ be $(s, p)$-harmonic in $D_R(\bar x)$ and satisfy $|v|\le C\ds$ in $D_R(\bar x)$. Then, there exist $0<\beta<1$, $C>0$ depending on the data, s.t.\ for all $0<r\le R$
\[\underset{D_r(\bar x)}{\rm osc}\,\frac{v}{\ds} \le C\Big\|\frac{v}{\ds}\Big\|_{L^\infty(D_R(\bar x))}\Big(\frac{r}{R}\Big)^\beta.\]
\end{proposition}

\noindent
We conclude with a weighted $L^\infty$-bound on the $(s, p)$-harmonic extension of a function, close to the boundary:

\begin{lemma}\label{hwe}
There exist $0<\tilde\rho<1$, $C>0$ depending on the data, s.t.\ for all $u\in\w$, $x_0\in\Omega_{\tilde\rho}$, $0<R<\tilde\rho$, the $(s, p)$-harmonic extension $v\in\w$ of $u$ in $D_R(x_0)$ satisfies
\[\Big\|\frac{v}{\ds}\Big\|_{L^\infty(D_R(x_0))} \le C\Big\|\frac{u}{\ds}\Big\|_{L^\infty(\Omega\setminus D_R(x_0))}.\]
\end{lemma}
\begin{proof}
Let $\bar\rho$ be given by Lemma \ref{geo} and assume, without loss of generality, that $0<\bar\rho<1$.  For $\eps$ given in Proposition \ref{bar}, define $\tilde\rho$, depending on the data, as
\[\tilde\rho = \frac{\eps\bar\rho}{4} \in (0,1).\]
Now fix $x_0\in\Omega_{\tilde\rho}$ and $0<R<\tilde\rho$, and assume $u/\ds\in L^\infty(\Omega\setminus D_R(x_0))$ (otherwise there is nothing to prove). We recall that $v\in\w$ is the unique solution of \eqref{hex}. Further, let $x_1\in D_R(x_0)$ be s.t.\ $v(x_1)\ge 0$ (alternatively, we work with $-u$, $-v$). Since $ \dd_\Omega$ is $1$-Lipschitz continuous, we have
\[\dd_\Omega(x_1) \le  \dd_\Omega(x_0)+|x_1-x_0| \le \tilde\rho+R < 2\tilde\rho,\]
in particular $ \dd_\Omega(x_1)<\bar\rho$. Then, Lemma \ref{geo} implies the existence of a unique $\bar x\in\partial\Omega$ s.t.\
\[|x_1-\bar x| =  \dd_\Omega(x_1) < 2\tilde\rho,\]
and a unique $y_1\in\Omega^c$ s.t.\ the line through $x_1$ and $y_1$ is normal to $\partial\Omega$ at $\bar x$, and
\[|y_1-\bar x| = \bar\rho = \frac{4\tilde\rho}{\eps}.\]
Therefore, the ball $B_{\bar\rho}(y_1)$ is exteriorly tangent to $\partial\Omega$ at $\bar x$ (see Figure \ref{fig3}).

\begin{figure}
\centering
\begin{tikzpicture}[scale=2.7]
\begin{scope}
\clip  (-2, -1)  .. controls (-1.2, 0) and (-0.5, 0) .. (0, 0) .. controls (0.5, 0) and (0.9, 0) .. (2, -0.5);
\filldraw[lightgray] (-0.1,-0.1) circle (0.24cm);
\draw (-0.1,-0.1) circle (0.24cm);
\end{scope}
\draw[very thin, dashed] (0, -0.18) -- (0, 1.1);
\draw(-1.3, -0.5) node{$\Omega$};
\clip (-2, -0.6) rectangle (2.15, 1.6);
\draw[very thick] (-2, -1)  .. controls (-1.2, 0) and (-0.5, 0) .. (0, 0) .. controls (0.5, 0) and (0.9, 0) .. (2, -0.5);
\draw (0, 1.1) circle (1.1cm);
\draw (0, 0) circle (0.4cm);
\draw (0, 0) node[above left]{$\bar x$} (0, 1.1) node[left]{$y_1$};
\filldraw (0, 0) circle (0.5pt) (0, 1.1) circle (0.5pt) (0, -0.18) node[left]{$x_1$} circle (0.5pt);
\draw[very thin] (0, 1.1) -- node[above, midway, sloped]{$\bar\rho$} (1, 1.55) (0, 0) -- node[above, midway, sloped]{$\eps\bar\rho$} (0.37, 0.15);
\end{tikzpicture}
\caption{The set $D_R(x_0)$ in gray and the other auxiliary points.}
\label{fig3}
\end{figure}
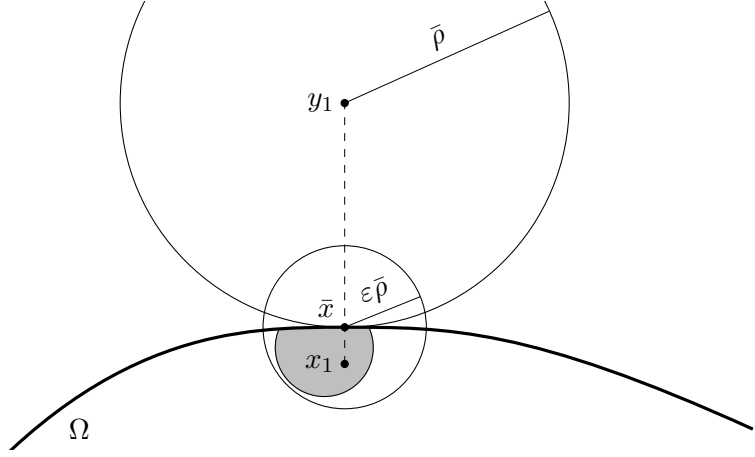

\vskip2pt
\noindent
Upon rescaling and translating the supersolution $w$ given in Proposition \ref{barrierina} we obtain a function  $\tilde w\in C^s(\R^N)\cap \widetilde{W}^{s,p}(B_{\eps\bar\rho}(\bar x)\setminus\overline{B}_{\bar\rho}(y_1))$ s.t.\
\beq\label{hwe1}
\begin{cases}
\fpl\tilde w > 0 & \text{in $B_{\eps\bar\rho}(\bar x)\setminus\overline{B}_{\bar\rho}(y_1)$} \\
\displaystyle\frac{\tilde\ds}{C} \le \tilde w \le C\tilde\ds & \text{in $\R^N$,}
\end{cases}
\eeq
where $C>1$ depends on the data and for all $x\in\R^N$ we have set
\[\tilde \dd_\Omega(x) = {\rm dist}(x,B_{\bar\rho}(y_1)),\]
in particular $\tilde \dd_\Omega(x_1)= \dd_\Omega(x_1)$ and $\tilde \dd_\Omega\ge \dd_\Omega$ in all of $\R^N$. Another geometrical observation is that, since $R<\tilde\rho$ and $ \dd_\Omega(x_1)<2\tilde\rho$, for all $x\in D_R(x_0)$ we have
\[|x-\bar x| \le |x-x_0|+|x_0-x_1|+|x_1-\bar x| \le 2R+ \dd_\Omega(x_1) < 4\tilde\rho=\eps\bar\rho,\]
that is, $D_R(x_0)\subseteq B_{\eps\bar\rho}(\bar x)\setminus\overline{B}_{\bar\rho}(y_1)$. Set now
\[\mu = C\Big\|\frac{u}{\ds}\Big\|_{L^\infty(\Omega\setminus D_R(x_0))} > 0,\]
with $C>1$ as in \eqref{hwe1}. We next show the function $\mu\tilde w\in\widetilde{W}^{s,p}(\Omega)$ is an upper barrier for $v$. Indeed, by \eqref{hex} \eqref{hwe1} we have in $D_R(x_0)$
\[\fpl(\mu\tilde w) > 0 = \fpl v,\]
while for all $x\in D_R(x_0)$
\[\mu\tilde w(x) \ge \frac{\mu\tilde\dd^s(x)}{C} \ge \Big\|\frac{u}{\ds}\Big\|_{L^\infty(\Omega\setminus D_R(x_0))}\ds(x) \ge u(x) = v(x).\]
Summarizing,
\[\begin{cases}
\fpl v \le \fpl(\mu\tilde w) & \text{in $D_R(x_0)$} \\
v \le \mu\tilde w & \text{in $D_R^c(x_0)$.}
\end{cases}\]
By Lemma \ref{wcp} we have $v\le\mu\tilde w$ in all of $\R^N$, in particular
\[v(x_1) \le \mu\tilde w(x_1) \le \mu C\tilde\ds(x_1) = C^2\Big\|\frac{u}{\ds}\Big\|_{L^\infty(\Omega\setminus D_R(x_0))}\ds(x_1).\]
By arbitrariness of $x_1$, it is proved that in $D_R(x_0)$
\[\frac{v}{\ds} \le C^2\Big\|\frac{u}{\ds}\Big\|_{L^\infty(\Omega\setminus D_R(x_0))},\]
while a symmetric estimate is proved using $-u$, $-v$. Thus, the conclusion is achieved.
\end{proof}

\section{Weighted H\"older regularity}\label{sec4}

\noindent
This section is devoted to the proof of Theorem \ref{fbr}. Our strategy consists in dividing the proof in two steps: first, we make the {\em qualitative} assumption that $f\in L^\infty(\Omega)$, and we use the known regularity theory for such case to obtain the core variance estimates on $u/\ds$, involving the norm of $f$ in $L^q(\Omega)$ ($q>N/s$); then, we remove such assumption and perform an approximation of $f$ via bounded functions, showing that the variance estimates previously obtained are stable. The ultimate estimate will provide the desired regularity.
\vskip2pt
\noindent
Our first results bounds the variance of $u/\ds$, defined as in \eqref{var}, by means of that of its $(s,p)$-harmonic extension, plus some perturbative terms:

\begin{lemma}\label{wve}
Let $f\in L^\infty(\Omega)$, $u\in\w$ be the solution of \eqref{dir}, $q>N/s$, $0<\gamma<1$ be defined as
\[\gamma = \frac{1}{\max\{p,2\}}\Big(s-\frac{N}{q}\Big).\]
Also, let $\bar x\in\overline\Omega$, $R>0$ be s.t.\ $ \dd_\Omega(\bar x)\le 2R$, and $v\in\w$ be the $(s, p)$-harmonic extension of $u$ in $D_R(\bar x)$. Then, there exists $C>0$ depending on the data and $q$, s.t.\ for all $x_0\in\Omega$, $r>0$ with $D_r(x_0)\subseteq D_R(\bar x)$, there holds
\[\sigma\Big(\frac{u}{\ds},x_0,r\Big) \le C\sigma\Big(\frac{v}{\ds},x_0,r\Big)+C\left[\Big\|\frac{u}{\ds}\Big\|_{L^\infty(\Omega)}^p+\|f\|_{L^q(\Omega)}^{p'}\right]R^{N+p\gamma}.\]
\end{lemma}
\begin{proof}
First note that, due to the assumption $f\in L^\infty(\Omega)$, by Theorem \ref{bhr} \ref{bhr2} we have $u/\ds\in C^\alpha(\overline\Omega)$ for some positive $\alpha$, in particular such quotient is bounded in $\Omega$. Now let $\tilde\rho\in(0,1)$ be as in Lemma \ref{hwe}.
\vskip2pt
\noindent
If $R\ge\tilde\rho/2$, then the conclusion is achieved immediately by the following argument. Whenever $D_r(x_0)\subseteq D_R(\bar x)$, we have
\begin{align*}
\sigma\Big(\frac{u}{\ds},x_0,r\Big) &\le C\Big\|\frac{u}{\ds}\Big\|_{L^\infty(\Omega)}^p|D_r(x_0)| \\
&\le C\Big\|\frac{u}{\ds}\Big\|_{L^\infty(\Omega)}^p\frac{R^{N+p\gamma}}{\tilde\rho^{p\gamma}} \\
&\le C\sigma\Big(\frac{v}{\ds},x_0,r\Big)+C\left[\Big\|\frac{u}{\ds}\Big\|_{L^\infty(\Omega)}^p+\|f\|_{L^q(\Omega)}^{p'}\right]R^{N+p\gamma},
\end{align*}
with $C>0$ depending on the data (as is $\tilde\rho$).
\vskip2pt
\noindent
Therefore, we assume henceforth $0<R<\tilde\rho/2$ and $D_r(x_0)\subseteq D_R(\bar x)$. By Lemma \ref{hwe}, we have
\beq\label{wve1}
\Big\|\frac{v}{\ds}\Big\|_{L^\infty(D_R(\bar x))} \le C\Big\|\frac{u}{\ds}\Big\|_{L^\infty(\Omega\setminus D_R(\bar x))},
\eeq
with $C>0$ depending on the data. Besides, by \eqref{propvar} we have
\beq\label{wve2}
\sigma\Big(\frac{u}{\ds},x_0,r\Big) \le C\sigma\Big(\frac{v}{\ds},x_0,r\Big)+C\int_{D_r(x_0)}\Big|\frac{u-v}{\ds}\Big|^p\,dx.
\eeq
The main step then consists in estimating the last term. To do so, we distinguish between the degenerate and singular regimes:
\begin{itemize}[leftmargin=1cm]
\item[$(a)$] If $p\ge 2$, then we begin by applying Theorem \ref{fhi} in $\Omega$, which is an admissible domain since it is bounded and with a $C^{1,1}$-boundary (see Lemma \ref{geo}). Next we apply Lemma \ref{mod}, test \eqref{dir} with $u-v\in\w$ and recall that $v$ is $(s, p)$-harmonic in $D_R(\bar x)$. Then, we use \eqref{wve1} and H\"older's inequality:
\begin{align*}
\int_{D_r(x_0)}\Big|\frac{u-v}{\ds}\Big|^p\,dx &\le C[u-v]_{s,p}^p \\
&\le C\langle\fpl u-\fpl v,u-v\rangle \\
&\le C\int_{D_R(\bar x)}|f||u-v|\,dx \\
&\le C\Big\|\frac{u-v}{\ds}\Big\|_{L^\infty(D_R(\bar x))}\|f\|_{L^1(D_R(\bar x))}R^s \\
&\le C\Big\|\frac{u}{\ds}\Big\|_{L^\infty(\Omega)}\|f\|_{L^q(\Omega)}|D_R(\bar x)|^\frac{q-1}{q}R^s \\
&\le C\Big\|\frac{u}{\ds}\Big\|_{L^\infty(\Omega)}\|f\|_{L^q(\Omega)}R^{N+(s-\frac{N}{q})}.
\end{align*}
Finally we apply Young's inequality with exponents $p$ and $p'$, so that the definition of $\gamma$ yields:
\[\int_{D_r(x_0)}\Big|\frac{u-v}{\ds}\Big|^p\,dx \le C\left[\Big\|\frac{u}{\ds}\Big\|_{L^\infty(\Omega)}^p+\|f\|_{L^q(\Omega)}^{p'}\right]R^{N+p\gamma}.\]
\item[$(b)$] If $p<2$, then we begin by applying Lemma \ref{mos} \ref{mos2} with $\Omega'=D_R(\bar x)$, noting that $D_r(x_0)\subseteq D_R(\bar x)$:
\[\int_{D_r(x_0)}\Big|\frac{u-v}{\ds}\Big|^p\,dx \le C\langle\fpl u-\fpl v,u-v\rangle^\frac{p}{2}\left[\Big\|\frac{u}{\ds}\Big\|_{L^p(D_R(\bar x))}^p+\Big\|\frac{v}{\ds}\Big\|_{L^p(D_R(\bar x))}^p\right]^\frac{2-p}{2}.\]
To estimate the first term we argue as in case $(a)$:
\[\langle\fpl u-\fpl v,u-v\rangle \le C\Big\|\frac{u}{\ds}\Big\|_{L^\infty(\Omega)}\|f\|_{L^q(\Omega)}R^{N+(s-\frac{N}{q})}.\]
For the second term we use \eqref{wve1}:
\begin{align*}
\Big\|\frac{u}{\ds}\Big\|_{L^p(D_R(\bar x))}^p+\Big\|\frac{v}{\ds}\Big\|_{L^p(D_R(\bar x))}^p &\le \left[\Big\|\frac{u}{\ds}\Big\|_{L^\infty(D_R(\bar x))}^p+\Big\|\frac{v}{\ds}\Big\|_{L^\infty(D_R(\bar x))}^p\right]|D_R(\bar x)|  \\
&\le C\Big\|\frac{u}{\ds}\Big\|_{L^\infty(\Omega)}^pR^N.
\end{align*}
Taking care of the powers and the definition of $\gamma$, we get the following inequality, to which we finally apply Young's inequality with exponents $2/(3-p)$, $2/(p-1)$:
\begin{align*}
\int_{D_r(x_0)}\Big|\frac{u-v}{\ds}\Big|^p\,dx &\le C\Big\|\frac{u}{\ds}\Big\|_{L^\infty(\Omega)}^\frac{p(3-p)}{2}\|f\|_{L^q(\Omega)}^\frac{p}{2}R^{N+\frac{p}{2}(s-\frac{N}{q})} \\
&\le C\left[\Big\|\frac{u}{\ds}\Big\|_{L^\infty(\Omega)}^p+\|f\|_{L^q(\Omega)}^{p'}\right]R^{N+p\gamma}.
\end{align*}
\end{itemize}
In both cases $(a)$ and $(b)$ we have proved the same estimate, which we plug into \eqref{wve2} to obtain the conclusion.
\end{proof}

\noindent
The aim of the next lemmas is to simplify the bound of Lemma \ref{wve}, by removing all terms on the right hand side except the norm of $f$. The first term we can erase is the variance of $v$, at the cost of replacing $\gamma$ with an undetermined exponent:

\begin{lemma}\label{nov}
Let $f\in L^\infty(\Omega)$, $u\in\w$ be the solution of \eqref{dir}, $q>N/s$. Then, there exist $0<\alpha<1$, $C>0$ depending on the data and $q$, s.t.\ for all $x_0\in\Omega$, $r>0$
\[\sigma\Big(\frac{u}{\ds},x_0,r\Big) \le C\left[\Big\|\frac{u}{\ds}\Big\|_{L^\infty(\Omega)}^p+\|f\|_{L^q(\Omega)}^{p'}\right]r^{N+p\alpha}.\]
\end{lemma}
\begin{proof}
First note that, due to $f\in L^\infty(\Omega)$ and Theorem \ref{bhr} \ref{bhr2}, all norms on the right hand side are finite. Let $\bar\rho$, $\tilde\rho$ be as in Lemmas \ref{geo}, \ref{hwe}, respectively. Without loss of generality we may assume $0<\tilde\rho\le\bar\rho<1$, and fix $0<R_0<\tilde\rho/4$ depending on the data, to be determined later. Now fix $x_0\in\Omega$, $r>0$.
\vskip2pt
\noindent
If $r\ge R_0$, then as in the proof of Lemma \ref{wve} we have
\[\sigma\Big(\frac{u}{\ds},x_0,r\Big) \le C\Big\|\frac{u}{\ds}\Big\|_{L^\infty(D_r(x_0))}^p|D_r(x_0)|\le \frac{C}{R_0^{p\alpha}}\Big\|\frac{u}{\ds}\Big\|_{L^\infty(D_r(x_0))}^p r^{N+p\alpha},\]
for any $0<\alpha<1$, implying the claimed estimate.
\vskip2pt
\noindent
We will henceforth assume $0<r<R_0$ and introduce a further radius $0<R<R_0$. We distinguish two cases:
\begin{itemize}[leftmargin=1cm]
\item[$(a)$] If $ \dd_\Omega(x_0)\ge 2R$, then $D_{2R}(x_0)=B_{2R}(x_0)$ and we are led back to interior regularity theory. By Theorem \ref{lhr}, for any $0<\delta<\bar\alpha$ with $\bar\alpha$ defined by \eqref{exp} we have $u\in C^\delta(\overline{B}_{R/2}(x_0))$. Besides, by the Lipschitz continuity of $ \dd_\Omega$, for all $x\in B_{R/2}(x_0)$ we have
\[| \dd_\Omega(x)- \dd_\Omega(x_0)| \le |x-x_0| \le \frac{R}{2} \le \frac{ \dd_\Omega(x_0)}{4},\]
which implies
\[\frac{ \dd_\Omega(x_0)}{4} \le  \dd_\Omega(x) \le \frac{3 \dd_\Omega(x_0)}{4}.\]
So we have $ \dd_\Omega^{-s}\in C^\delta(\overline{B}_{R/2}(x_0))$ with
\[[ \dd_\Omega^{-s}]_{C^\delta(\overline{B}_{R/2}(x_0))} \le CR^{1-\delta}\sup_{B_{R/2}(x_0)}\,\frac{|\nabla \dd_\Omega|}{ \dd_\Omega^{1+s}} \le \frac{CR^{1-\delta}}{ \dd_\Omega^{1+s}(x_0)} \le \frac{C}{\ds(x_0)R^\delta}.\]
Combining the estimates above, we have
\begin{align*}
\left[\frac{u}{\ds}\right]_{C^\delta(\overline{B}_{R/2}(x_0))} &\le C[u]_{C^\delta(\overline{B}_{R/2}(x_0))}\| \dd_\Omega^{-s}\|_{L^\infty(B_{R/2}(x_0))}+C\|u\|_{L^\infty(B_{R/2}(x_0))}[ \dd_\Omega^{-s}]_{C^\delta(\overline{B}_{R/2}(x_0))} \\
&\le \frac{C}{\ds(x_0)}[u]_{C^\delta(\overline{B}_{R/2}(x_0))}+\frac{C}{R^\delta\ds(x_0)}\|u\|_{L^\infty(B_{R/2}(x_0))} \\
&\le \frac{C}{\ds(x_0)}[u]_{C^\delta(\overline{B}_{R/2}(x_0))}+\frac{C}{R^\delta}\Big\|\frac{u}{\ds}\Big\|_{L^\infty(B_{R/2}(x_0))}.
\end{align*}
To estimate the first term, we apply Corollary \ref{chr} with $\alpha=0$ (that is, $H_0(u)=\|u\|_{L^\infty(\Omega)}$ by \eqref{hal}), then again the two-sided pointwise bound on $ \dd_\Omega$:
\begin{align*}
\frac{[u]_{C^\delta(\overline{B}_{R/2}(x_0))}}{\ds(x_0)} &\le \frac{C}{R^\delta}\left[\frac{\|u\|_{L^\infty(\Omega)}}{\ds(x_0)}+\frac{R^{p'(s-\frac{N}{pq})}}{\ds(x_0)}\|f\|_{L^q(B_R(x_0))}^\frac{1}{p-1}\right] \\
&\le \frac{C}{R^\delta}\left[\Big\|\frac{u}{\ds}\Big\|_{L^\infty(\Omega)}+\Big(\frac{R}{\ds(x_0)}\Big)^s\|f\|_{L^q(B_R(x_0))}^\frac{1}{p-1}\right] \\
&\le \frac{C}{R^\delta}\left[\Big\|\frac{u}{\ds}\Big\|_{L^\infty(\Omega)}+\|f\|_{L^q(B_R(x_0))}^\frac{1}{p-1}\right],
\end{align*}
where we have used that $R<1$ and
\[p'\Big(s-\frac{N}{pq}\Big) > s.\]
Connecting with the inequality above, we get
\[\left[\frac{u}{\ds}\right]_{C^\delta(\overline{B}_{R/2}(x_0))} \le \frac{C}{R^\delta}\left[\Big\|\frac{u}{\ds}\Big\|_{L^\infty(\Omega)}+\|f\|_{L^q(\Omega)}^\frac{1}{p-1}\right].\]
Passing to the variance, we have for all $0<r<R/2$
\begin{align}\label{nov1}
\sigma\Big(\frac{u}{\ds},x_0,r\Big) &\le C\left[\frac{u}{\ds}\right]_{C^\delta(\overline{B}_{R/2}(x_0))}^pr^{N+p\delta} \\
\nonumber &\le C\left[\Big\|\frac{u}{\ds}\Big\|_{L^\infty(\Omega)}^p+\|f\|_{L^q(\Omega)}^{p'}\right]\Big(\frac{r}{R}\Big)^{p\delta}r^N.
\end{align}
\item[$(b)$] If $ \dd_\Omega(x_0)<2R$, then we exploit boundary regularity theory and $(s, p)$-harmonic extensions. Fix $\bar x\in\partial\Omega$ s.t.\ $|x_0-\bar x|= \dd_\Omega(x_0)$, $0<r<R/2$, so for all $x\in D_r(x_0)$
\[|x-\bar x| \le |x-x_0|+|x_0-\bar x| < r+ \dd_\Omega(x_0) \le 3R,\]
i.e., $D_r(x_0)\subseteq D_{r+ \dd_\Omega(x_0)}(x_0)\subseteq D_{3R}(\bar x)$. Let $v\in\w$ solve
\[\begin{cases}
\fpl v = 0 & \text{in $D_{3R}(\bar x)$} \\
v = u & \text{in $D_{3R}^c(\bar x)$.}
\end{cases}\]
First, we concatenate Proposition \ref{boe} and Lemma \ref{hwe} along with some domain inclusions to get, for some $0<\beta<1$ depending on the data,
\[
\underset{D_r(x_0)}{\rm osc}\,\frac{v}{\ds} \le C\Big\|\frac{v}{\ds}\Big\|_{L^\infty(D_{3R}(x_0))}\Big(\frac{r+ \dd_\Omega(x_0)}{3R}\Big)^\beta \le C\Big\|\frac{u}{\ds}\Big\|_{L^\infty(\Omega)}\Big(\frac{r+ \dd_\Omega(x_0)}{R}\Big)^\beta.
\]
As usual, we bound the variance via the oscillation:
\begin{align*}
\sigma\Big(\frac{v}{\ds},x_0,r\Big) &\le C\Big(\underset{D_r(x_0)}{\rm osc}\,\frac{v}{\ds}\Big)^p|D_r(x_0)| \\
&\le C\Big\|\frac{u}{\ds}\Big\|_{L^\infty(\Omega)}^p\Big(\frac{r+ \dd_\Omega(x_0)}{R}\Big)^{p\beta}r^N.
\end{align*}
Next we apply Lemma \ref{wve} (with center $\bar x$ and radius $3R$, up to a constant rescaling) and the previous variance estimate on $v$, to find for all $0<r<R/2$
\begin{align}\label{nov2}
\sigma\Big(\frac{u}{\ds},x_0,r\Big) &\le C\sigma\Big(\frac{v}{\ds},x_0,r\Big)+C\left[\Big\|\frac{u}{\ds}\Big\|_{L^\infty(\Omega)}^p+\|f\|_{L^q(\Omega)}^{p'}\right]R^{N+p\gamma} \\
\nonumber &\le C\Big\|\frac{u}{\ds}\Big\|_{L^\infty(\Omega)}^p\Big(\frac{r+ \dd_\Omega(x_0)}{R}\Big)^{p\beta}r^N+C\left[\Big\|\frac{u}{\ds}\Big\|_{L^\infty(\Omega)}^p+\|f\|_{L^q(\Omega)}^{p'}\right]R^{N+p\gamma}.
\end{align}
\end{itemize}
We need to merge \eqref{nov1} and \eqref{nov2} into one estimate, holding in all cases. To do so, we will chose $R>2r$ according to different situations. Fix $0<\eta<\theta\le 1$ to be determined later, and let $x_0\in\Omega$, $0<r<1$. Again we distinguish two cases:
\begin{itemize}[leftmargin=1cm]
\item[$(a')$] If $ \dd_\Omega(x_0)\ge 4r^\theta$, then we set $R=2r^\theta$ so that
\[0 < r < r^\theta = \frac{R}{2} \le \frac{ \dd_\Omega(x_0)}{4},\]
and we are in case $(a)$ of the previous alternative. So, by \eqref{nov1} we have
\beq\label{nov3}
\sigma\Big(\frac{u}{\ds},x_0,r\Big) \le C\left[\Big\|\frac{u}{\ds}\Big\|_{L^\infty(\Omega)}^p+\|f\|_{L^q(\Omega)}^{p'}\right]r^{N+p\delta(1-\theta)},
\eeq
with $0<\delta<\bar\alpha$ and $\bar\alpha$ given by \eqref{exp}.
\item[$(b')$] If $ \dd_\Omega(x_0)<4r^\theta$, then we set $R=2r^\eta$, so that $0<r<R/2$, $ \dd_\Omega(x_0)< 2R$, and without loss of generality we may assume $R<\tilde\rho/4$. We are then reduced to case $(b)$ above, hence from \eqref{nov2} and the relations between radii we have
\beq\label{nov4}
\sigma\Big(\frac{u}{\ds},x_0,r\Big) \le C\Big\|\frac{u}{\ds}\Big\|_{L^\infty(\Omega)}^p r^{N+p\beta(\theta-\eta)}+C\left[\Big\|\frac{u}{\ds}\Big\|_{L^\infty(\Omega)}^p+\|f\|_{L^q(\Omega)}^{p'}\right]r^{\eta(N+p\gamma)},
\eeq
with $0<\beta<1$ given by Proposition \ref{boe} and $0<\gamma<1$ given by Lemma \ref{wve}.
\end{itemize}
In order to gather all alternatives, we seek $\eta$, $\theta$ s.t.\ all powers of $r$ in \eqref{nov3}, \eqref{nov4}, respectively, agree, that is,
\[N+p\delta(1-\theta) = N+p\beta(\theta-\eta) = \eta(N+p\gamma).\]
From the equalities above, through a straightforward computation we infer
\[\eta = \frac{N(\beta+\delta)+p\beta\delta}{(N+p\gamma)(\beta+\delta)+p\beta\delta}, \ \theta = \frac{(N+p\gamma)\delta+(N+p\delta)\beta}{(N+p\gamma)(\beta+\delta)+p\beta\delta}.\]
Let us check that such choice is admissible. From $\gamma>0$ we clearly have $0<\eta<1$, which in turn implies
\[\theta = \frac{\beta\eta+\delta}{\beta+\delta} \in (\eta,1).\]
Therefore, set $\alpha=\delta(1-\theta)\in(0,1)$ and take $C>0$ as the biggest of all $C'$s in \eqref{nov3}, \eqref{nov4} (both only depend on the data and $q$). So we have for all $x_0\in\Omega$ and $0<r<1$
\[\sigma\Big(\frac{u}{\ds},x_0,r\Big) \le C\left[\Big\|\frac{u}{\ds}\Big\|_{L^\infty(\Omega)}^p+\|f\|_{L^q(\Omega)}^{p'}\right]r^{N+p\alpha},\]
which concludes the proof.
\end{proof}

\noindent
We already know that $u/\ds\in C^\alpha(\overline\Omega)$ by the assumed boundedness of the reaction, but the inequality in the previous lemma solely involves the $L^q(\Omega)$ norm of $f$. In order to  transfer the regularity to the general case $f\in L^q(\Omega)$,  it remains to remove the dependance on $\|u/\ds\|_\infty$ in the variance estimate, which is the purpose of the next lemma:

\begin{lemma}\label{nou}
Let $f\in L^\infty(\Omega)$, $u\in\w$ be the solution of \eqref{dir}, $q>N/s$. Then:
\begin{enumroman}
\item\label{nou1} there exists $C>0$ depending on the data and $q$, s.t.\
\[\Big\|\frac{u}{\ds}\Big\|_{L^\infty(\Omega)} \le C\|f\|_{L^q(\Omega)}^\frac{1}{p-1};\]
\item\label{nou2} there exist $0<\alpha<1$, $C>0$ depending on the data and $q$, s.t.\ for all $x_0\in\Omega$, $r>0$
\[\sigma\Big(\frac{u}{\ds},x_0,r\Big) \le C\|f\|_{L^q(\Omega)}^{p'}r^{N+p\alpha}.\]
\end{enumroman}
\end{lemma}
\begin{proof}
As already noted in Section \ref{sec2}, from $q>N/s$ we have $q'<\p$, hence $\w\hookrightarrow L^{q'}(\Omega)$. Testing \eqref{dir} with $u\in\w$ and using H\"older's and Sobolev's inequalities, we have
\[
[u]_{s,p}^p = \langle\fpl u,u\rangle = \int_\Omega fu\,dx \le \|f\|_{L^q(\Omega)}\|u\|_{L^{q'}(\Omega)} \le C\|f\|_{L^q(\Omega)}[u]_{s,p},
\]
hence
\beq\label{nou3}
[u]_{s,p} \le C\|f\|_{L^q(\Omega)}^\frac{1}{p-1}.
\eeq
By regularity of $\partial\Omega$, for all $x_0\in\Omega$ and $r>0$ we have $|D_r(x_0)|\ge\mu r^N$, with $\mu>0$ depending on $\Omega$. So, we can apply Theorem \ref{cam} to $u/\ds$ in $\Omega$, which, combined with Lemma \ref{nov}, yields for some $0<\alpha<1$, $C>0$ depending on the data
\beq\label{nou4}
\left[\frac{u}{\ds}\right]_{C^\alpha(\overline\Omega)}^p \le C\left[\Big\|\frac{u}{\ds}\Big\|_{L^\infty(\Omega)}^p+\|f\|_{L^q(\Omega)}^{p'}\right].
\eeq
Let us prove \ref{nou1}. Fix $x_0\in\Omega$, $0<\eps<1$. By \eqref{nou4}, for all $x\in D_\eps(x_0)$ we have
\begin{align*}
\Big|\frac{u(x_0)}{\ds(x_0)}\Big|^p &\le C\Big|\frac{u(x_0)}{\ds(x_0)}-\frac{u(x)}{\ds(x)}\Big|^p+C\Big|\frac{u(x)}{\ds(x)}\Big|^p \\
&\le C\left[\frac{u}{\ds}\right]_{C^\alpha(\overline\Omega)}^p\eps^{p\alpha}+C\Big|\frac{u(x)}{\ds(x)}\Big|^p \\
&\le C\left[\Big\|\frac{u}{\ds}\Big\|_{L^\infty(\Omega)}^p+\|f\|_{L^q(\Omega)}^{p'}\right]\eps^{p\alpha}+C\Big|\frac{u(x)}{\ds(x)}\Big|^p.
\end{align*}
We estimate the last term by passing to the mean value in $D_\eps(x_0)$. Being $\Omega$ bounded, we have $R_\Omega<\infty$, while by Lemma \ref{geo} $\Omega$ satisfies ${\rm EBC}(\bar\rho)$, so we can apply Theorem \ref{fhi}:
\[
\fint_{D_\eps(x_0)}\Big|\frac{u}{\ds}\Big|^p\,dx \le \frac{C}{\eps^N}\int_\Omega\Big|\frac{u}{\ds}\Big|^p\,dx \le \frac{C}{\eps^N}[u]_{s,p}^p \le \frac{C}{\eps^N}\|f\|_{L^q(\Omega)}^{p'},
\]
where in the last passage we have used \eqref{nou3}. Next, take the supremum with respect to $x_0\in\Omega$:
\[\Big\|\frac{u}{\ds}\Big\|_{L^\infty(\Omega)}^p \le C\Big\|\frac{u}{\ds}\Big\|_{L^\infty(\Omega)}^p\eps^{p\alpha}+C\|f\|_{L^q(\Omega)}^{p'}\Big(\eps^{p\alpha}+\frac{1}{\eps^N}\Big).\]
Choosing $\eps$ small enough, we reabsorb the first term on the right hand side and achieve \ref{nou1}.
Finally, \ref{nou2} follows immediately from \eqref{nou4} and \ref{nou1}.
\end{proof}

\noindent
We can now prove our fine boundary regularity result.
\vskip4pt
\noindent
{\em Proof of Theorem \ref{fbr}.} Since $f\in L^q(\Omega)$ with $q>N/s$, by a standard truncation argument we can find a sequence $(f_n)$ in $L^\infty(\Omega)$ s.t.\ $f_n\to f$ in $L^q(\Omega)$ and
\[\|f_n\|_{L^q(\Omega)} \le \|f\|_{L^q(\Omega)}.\]
For all $n\in\N$ there exist a unique $u_n\in\w$ solving
\beq\label{fbr1}
\begin{cases}
\fpl u_n = f_n & \text{in $\Omega$} \\
u_n = 0 & \text{in $\Omega^c$.}
\end{cases}
\eeq
By Theorem \ref{bhr} \ref{bhr2}, there exists $0<\alpha\le s$ independent of $n$, s.t.\ $u_n/\ds\in C^\alpha(\overline\Omega)$ (up to extension to $\overline\Omega$). Besides, by Lemma \ref{nou} \ref{nou2}, up to taking $\alpha$ even smaller, we can find $C>0$ independent of $n$ s.t.\ for all $x_0\in\Omega$, $r>0$
\[\sigma\Big(\frac{u_n}{\ds},x_0,r\Big) \le C\|f_n\|_{L^q(\Omega)}^{p'}r^{N+p\alpha} \le C\|f\|_{L^q(\Omega)}^{p'}r^{N+p\alpha}.\]
We apply Theorem \ref{cam} (recalling the geometrical observation made above), so that the qualitative information provided by Theorem \ref{bhr} is coupled with the following uniform estimate:
\beq\label{fbr2}
\left[\frac{u_n}{\ds}\right]_{C^\alpha(\overline\Omega)} \le C\|f\|_{L^q(\Omega)}^\frac{1}{p-1}.
\eeq
There remain to pass to the limit in \eqref{fbr2} as $n\to\infty$. With this aim in mind, as in the proof of Lemma \ref{nou} we test \eqref{fbr1} with $u_n$ and find
\[[u_n]_{s,p} \le C\|f_n\|_{L^q(\Omega)}^\frac{1}{p-1} \le C\|f\|_{L^q(\Omega)}^\frac{1}{p-1}.\]
The sequence $(u_n)$ is thus bounded in $\w$, hence by reflexivity we can pass to a subsequence s.t.\ $u_n\rightharpoonup u$ in $\w$. Due to the compact embedding $\w\hookrightarrow L^{q'}(\Omega)$, passing to a further subsequence we have $u_n\to u$ in $L^{q'}(\Omega)$ and $u_n(x)\to u(x)$ for a.e.\ $x\in\Omega$. Test again \eqref{fbr1}, this time with $u_n-u\in\w$:
\begin{align*}
\langle\fpl u_n,u_n-u\rangle &= \int_\Omega f_n(u_n-u)\,dx \\
&\le \|f_n\|_{L^q(\Omega)}\|u_n-u\|_{L^{q'}(\Omega)} \le \|f\|_{L^q(\Omega)}\|u_n-u\|_{L^{q'}(\Omega)},
\end{align*}
and the latter tends to $0$ as $n\to\infty$. Hence,
\[\limsup_n\,\langle\fpl u_n,u_n-u\rangle \le 0.\]
By the $(S)_+$-property of $\fpl$ (see \cite[Lemma 5.1]{FI}), we have $u_n\to u$ in $\w$. We can then pass to the limit in \eqref{fbr1} and see that $u\in\w$ coincides with the unique solution of \eqref{dir}.
\vskip2pt
\noindent
By \eqref{fbr2}, the sequence $(u_n/\ds)$ is bounded in $C^\alpha(\overline\Omega)$. By the Ascoli-Arzel\`a theorem, passing to a subsequence we have $u_n/\ds\to u/\ds$ uniformly in $\overline\Omega$ (by uniqueness), hence we infer $u/\ds\in C^\alpha(\overline\Omega)$ and
\[\left[\frac{u}{\ds}\right]_{C^\alpha(\overline\Omega)} \le C\|f\|_{L^q(\Omega)}^\frac{1}{p-1}.\]
In addition, by Lemma \ref{nou} \ref{nou1} we have for all $n\in\N$
\[\Big\|\frac{u_n}{\ds}\Big\|_{L^\infty(\Omega)} \le C\|f_n\|_{L^q(\Omega)}^\frac{1}{p-1},\]
with $C>0$ independent of $n$. Passing to the limit by uniform convergence, and recalling that $f_n\to f$ in $L^q(\Omega)$, we find
\[\Big\|\frac{u}{\ds}\Big\|_{L^\infty(\Omega)} \le C\|f\|_{L^q(\Omega)}^\frac{1}{p-1}.\]
Thus, for some $C_\alpha>0$ depending on the data and $q$, we have
\[\Big\|\frac{u}{\ds}\Big\|_{C^\alpha(\overline\Omega)} \le C_\alpha\|f\|_{L^q(\Omega)}^\frac{1}{p-1},\]
which concludes the proof. \qed

\section{Global H\"older regularity}\label{sec5}

\noindent
This final section is devoted to the proof of Theorem \ref{hol}. The case $q>N/s$ will follow from Theorem \ref{fbr} and Lemma \ref{rsb}, while for the case $N/ps<q\le N/s$ we need to repeat the variance estimates of Section \ref{sec3}, with appropriate adaptations. Note that, in this framework, we cannot rely on the case of bounded reactions.
\vskip2pt
\noindent
We begin with a first estimate, analogous to Lemma \ref{wve}, but with a different exponent $\gamma$:

\begin{lemma}\label{qve}
Let $f\in L^q(\Omega)$ with $q\ge 1$, $q>N/ps$, $u\in\w$ be the solution of \eqref{dir}, $0<\gamma<1$ be defined as
\[\gamma = \min\Big\{1,\,\frac{p}{2}\Big\}\Big(s-\frac{N}{pq}\Big).\]
Also, let $\bar x\in\partial\Omega$, $R>0$, and $v\in\w$ be the $(s, p)$-harmonic extension of $u$ in $D_R(\bar x)$. Then, there exists $C>0$ depending on the data and $q$, s.t.\ for all $x_0\in\Omega$, $0<r<R$ with $D_r(x_0)\subseteq D_R(\bar x)$
\begin{enumroman}
\item\label{qve1} if $p\ge 2$, then
\[\sigma(u,x_0,r) \le C\sigma(v,x_0,r)+C\|u-v\|_{L^\infty(D_R(\bar x))}\|f\|_{L^q(\Omega)}R^{N+p\gamma};\]
\item\label{qve2} if $p<2$, then
\[\sigma(u,x_0,r) \le C\sigma(v,x_0,r)+C\big[\|u\|_{L^\infty(D_R(\bar x))}+\|v\|_{L^\infty(D_R(\bar x))}\big]^\frac{p(3-p)}{2}\|f\|_{L^q(\Omega)}^\frac{p}{2}R^{N+p\gamma}.\]
\end{enumroman}
\end{lemma}
\begin{proof}
Without loss of generality, we may assume that all norms are finite. Fix $x_0\in\Omega$, $0<r<R$ s.t.\ $D_r(x_0)\subseteq D_R(\bar x)$ so that, by \eqref{propvar}, it suffices to estimate the quantity
\[\int_{D_R(\bar x)}|u-v|^p\,dx.\]
To this end we distinguish the degenerate and singular regimes, respectively:
\begin{itemize}[leftmargin=1cm]
\item[$(a)$] If $p\ge 2$, then we note that $D_R(\bar x)=\Omega\cap B_R(\bar x)$ satisfies ${\rm EBC}(\bar\rho)$ with $\bar\rho$ defined by Lemma \ref{geo} (see Remark \ref{ecb}). Note that $\dd_{D_R(\bar x)}\le R$ in all of $\R^N$, and apply Theorem \ref{fhi} to $(u-v)\in W^{s,p}_0(D_R(\bar x))$:
\[\int_{D_R(\bar x)}|u-v|^p\,dx \le R^{ps}\int_{D_R(\bar x)}\frac{|u(x)-v(x)|^p}{  \dd_{D_R(\bar x)}^{ps}(x)}\,dx \le CR^{ps}[u-v]_{s,p}^p.\]
To estimate the Gagliardo norm, we apply Lemma \ref{mod}, then test \eqref{dir} with $(u-v)$ and use H\"older's inequality:
\begin{align*}
[u-v]_{s,p}^p &\le C\langle\fpl u-\fpl v,u-v\rangle \\
&= C\int_{D_R(\bar x)}f(u-v)\,dx \\
&\le C\|u-v\|_{L^\infty(D_R(\bar x))}\|f\|_{L^1(D_R(\bar x))} \\
&\le C\|u-v\|_{L^\infty(D_R(\bar x))}\|f\|_{L^q(\Omega)}R^{N(1-\frac{1}{q})}.
\end{align*}
Plugging the estimates above together, and noting that
\[ps+N\Big(1-\frac{1}{q}\Big) = N+p\gamma,\]
we have for some $C>0$ depending on the data and $q$
\[\int_{D_R(\bar x)}|u-v|^p\,dx \le C\|u-v\|_{L^\infty(D_R(\bar x))}\|f\|_{L^q(\Omega)}R^{N+p\gamma}.\]
This, along with \eqref{propvar}, proves \ref{qve1}.
\item[$(b)$] If $p<2$, then the monotonicity property is a subtler one. We recall that $\Omega$ has finite inradius and satisfies ${\rm EBC}(\bar\rho)$, hence we apply Lemma \ref{mos} \ref{mos1} with $\Omega'=D_R(\bar x)$:
\begin{align*}
\int_{D_R(\bar x)}|u-v|^p\,dx &\le C\| \dd_\Omega\|_{L^\infty(D_R(\bar x))}^\frac{p^2s}{2}\langle\fpl u-\fpl v,u-v\rangle^\frac{p}{2} \\
&\cdot \big[\|u\|_{L^p(D_R(\bar x))}^p+\|v\|_{L^p(D_R(\bar x))}^p\big]^\frac{2-p}{2},
\end{align*}
with $C>0$ depending on $\Omega$ and $\dd_\Omega$ still denoting the distance from $\Omega^c$. To estimate the first factor, we simply note that $\dd_\Omega\le R$ in $D_R(\bar x)$. For the central factor, again we use testing and H\"older's inequality:
\begin{align*}
\langle\fpl u-\fpl v,u-v\rangle &= \int_{D_R(\bar x)}f(u-v)\,dx \\
&\le \|u-v\|_{L^\infty(D_R(\bar x))}\|f\|_{L^1(D_R(\bar x))} \\
&\le \big[\|u\|_{L^\infty(D_R(\bar x))}+\|v\|_{L^\infty(D_R(\bar x))}\big]\|f\|_{L^q(D_R(\bar x))}|D_R(\bar x)|^\frac{q-1}{q} \\
&\le C\big[\|u\|_{L^\infty(D_R(\bar x))}+\|v\|_{L^\infty(D_R(\bar x))}\big]\|f\|_{L^q(\Omega)}R^{N(1-\frac{1}{q})}.
\end{align*}
The last factor above is easily estimated as follows:
\[\|u\|_{L^p(D_R(\bar x))}^p+\|v\|_{L^p(D_R(\bar x))}^p \le C\big[\|u\|_{L^\infty(D_R(\bar x))}+\|v\|_{L^\infty(D_R(\bar x))}\big]^pR^N.\]
Plugging back these inequalities into the previous one, we have for some $C>0$ depending on the data and $q$
\[\int_{D_R(\bar x)}|u-v|^p\,dx \le C\big[\|u\|_{L^\infty(D_R(\bar x))}+\|v\|_{L^\infty(D_R(\bar x))}\big]^\frac{p(3-p)}{2}\|f\|_{L^q(\Omega)}^\frac{p}{2}R^{N(1-\frac{p}{2q})+\frac{p^2s}{2}}.\]
Since 
\[N\Big(1-\frac{p}{2q}\Big)+\frac{p^2s}{2} = N+p\gamma,\]
the inequality above rephrases as
\[\int_{D_R(\bar x)}|u-v|^p\,dx \le C \left[\|u\|_{L^\infty(D_R(\bar x))}+\|v\|_{L^\infty(D_R(\bar x))}\right]^\frac{p(3-p)}{2}\|f\|_{L^q(\Omega)}^\frac{p}{2}R^{N+p\gamma},\]
which, recalling \eqref{propvar} again, provides \ref{qve2}.
\end{itemize}
In either case, the proof is concluded.
\end{proof}

\noindent
The next step differs substantially from those seen in Section \ref{sec3}, as we aim at optimal H\"older exponents. This will first require a basic variance estimate, and then an iterative procedure {\em \'a la} Moser:

\begin{lemma}\label{ite}
Let $f\in L^q(\Omega)$ with $q\ge 1$, $q>N/ps$, $u\in\w$ be the solution of \eqref{dir}, and
\[0 < \alpha < \min\Big\{s,\,p'\Big(s-\frac{N}{pq}\Big)\Big\}.\]
Then, $u\in C^\alpha(\overline\Omega)$ and there exists $C_\alpha>0$ depending on the data, $q$, and $\alpha$, s.t.\
\[\|u\|_{C^\alpha(\overline\Omega)} \le C_\alpha\|f\|_{L^q(\Omega)}^\frac{1}{p-1}.\]
\end{lemma}
\begin{proof}
Due to $(p-1)$-homogeneity of $\fpl$, without loss of generality we may assume $\|f\|_{L^q(\Omega)}=1$. By \cite[Theorem 3.1]{BP}, there exists $C_0>0$ depending on the data and $q$, s.t.\
\beq\label{ite1}
\|u\|_{L^\infty(\Omega)} \le C_0.
\eeq
Let $\bar\alpha\in(0,1]$ be defined by \eqref{exp}, and $0<\delta<\bar\alpha$ be fixed. For all $\alpha\in [0, 1)$ define $H_\alpha(u)$ as in \eqref{hal}. By \eqref{ite1} we know that $H_0(u)<\infty$. We claim that $H_\alpha(u)<\infty$ for all $0\le\alpha<\min\{s,\delta\}$, with a uniform estimate. To prove such claim, we will construct an iterative scheme on $H_\alpha(u)$, with $\alpha$ lying in the desired interval.
\vskip2pt
\noindent
Let us fix $0\le\alpha<\min\{s,\delta\}$, and assume $H_\alpha(u)<\infty$. Also, let us fix $x_0\in\Omega$, $0<r<1$ small enough (to be determined later). We consider a radius $R>r$ (subject to further conditions to be detailed later), and as in Lemma \ref{nov} we distinguish two cases:
\begin{itemize}[leftmargin=1cm]
\item[$(a)$] If $ \dd_\Omega(x_0)\ge 2R$, then $D_{2R}(x_0)=B_{2R}(x_0)$ and we can apply local regularity theory. By Corollary \ref{chr}, recalling that $\|f\|_{L^q(\Omega)}=1$ and the bounds on $\alpha$, we have $u\in C^\delta(\overline{B}_{R/2}(x_0))$ and there exists $C>0$ depending on the data, $q$, $\delta$, and $\alpha$ s.t.\
\begin{align*}
[u]_{C^\delta(\overline{B}_{R/2}(x_0))} &\le \frac{C}{R^\delta}\left[H_\alpha(u) R^\alpha+R^{p'(s-\frac{N}{pq})}\right] \\
&\le \frac{C}{R^\delta}\left[H_\alpha(u)+{\rm diam}(\Omega)^{p'(s-\frac{N}{pq})-\alpha}\right]R^\alpha \le C\big[H_\alpha(u)+1\big]R^{\alpha-\delta}.
\end{align*}
This, in turn, implies the following variance estimate for all $0<r\le R/2\le \dd_\Omega(x_0)/4$:
\begin{equation}\label{ite2}
\sigma(u,x_0,r) \le C[u]_{C^\delta(\overline{B}_{R/2}(x_0))}^pr^{N+p\delta}\le C\big[H_\alpha^p(u)+1\big]R^{p(\alpha-\delta)}r^{N+p\delta}.
\end{equation}
\item[$(b)$] If $ \dd_\Omega(x_0)<2R$, then we assume $R\le\bar\rho/4$, with $0<\bar\rho<1$ as in Lemma \ref{geo}. Hence, there exists a unique $\bar x\in\partial\Omega$ s.t.\ $|x_0-\bar x|= \dd_\Omega(x_0)$, and for all $0<r\le R$ we have
\[D_r(x_0) \subset D_{r+ \dd_\Omega(x_0)}(\bar x) \subseteq D_{3R}(\bar x).\]
Let $v\in\w$ be the $(s, p)$-harmonic extension of $u$ in $D_{3R}(\bar x)$, that is, the unique solution of
\[\begin{cases}
\fpl v = 0 & \text{in $D_{3R}(\bar x)$} \\
v = u & \text{in $D_{3R}^c(\bar x)$.}
\end{cases}\]
We apply Lemma \ref{qve} in $D_{3R}(\bar x)$, recalling that $\|f\|_{L^q(\Omega)}=1$ and $ \dd_\Omega\le 3R$ in $D_{3R}(\bar x)$. Setting for simplicity of notation
\[\xi_p = \max\Big\{1,\,\frac{p(3-p)}{2}\Big\}\]
allows to summarize the two cases of Lemma \ref{qve} as
\[\sigma(u,x_0,r) \le C\sigma(v,x_0,r)+C\big[\|u\|_{L^\infty(D_{3R}(\bar x))}+\|v\|_{L^\infty(D_{3R}(\bar x))}+1\big]^{\xi_p}R^{N+p\gamma} ,\]
where $C>0$ depends on the data and $q$, and we recall that
\[\gamma = \min\Big\{1,\frac{p}{2}\Big\}\Big(s-\frac{N}{pq}\Big).\]
We estimate separately both terms on the right hand side of the previous inequality. First note that $\alpha<p's$. Applying Lemma \ref{hoe} in $D_{3R}(\bar x)$, we get
\[\underset{D_r(x_0)}{\rm osc}\,v \le \underset{D_{r+ \dd_\Omega(x_0)}(\bar x)}{\rm osc}\,v \le CH_\alpha(u)\Big(\frac{r+ \dd_\Omega(x_0)}{R}\Big)^sR^\alpha\]
which  directly implies
\[\sigma(v,x_0,r) \le CH_\alpha^p(u)\Big(\frac{r+ \dd_\Omega(x_0)}{R}\Big)^{ps}R^{p\alpha}r^N.\]
Besides, applying Lemma \ref{abh} in $D_{3R}(\bar x)$ and recalling that $u(\bar x)=0$, we have
\[\|u\|_{L^\infty(D_{3R}(\bar x))}+\|v\|_{L^\infty(D_{3R}(\bar x))} \le C\|u\|_{L^\infty(D_{6R}(\bar x))}+C[u]_{C^\alpha(\overline\Omega)}R^\alpha \le CH_\alpha(u) R^\alpha.\]
Going back to the main inequality, we have for all $0<r<R$, $ \dd_\Omega(x_0)<2R\le\bar\rho/2$
\beq\label{ite3}
\sigma(u,x_0,r) \le CH_\alpha^p(u)\Big(\frac{r+ \dd_\Omega(x_0)}{R}\Big)^{ps}R^{p\alpha}r^N+C\big[H_\alpha^{\xi_p}(u)+1\big]R^{N+\xi_p\alpha+p\gamma}.
\eeq
\end{itemize}
In order to merge \eqref{ite2} and \eqref{ite3} into one estimate, we argue as in Lemma \ref{nov}. Fix $x_0\in\Omega$, $0<r<1$ to be chosen conveniently small, and $0<\eta<\theta\le 1$ to be determined later. Again we distinguish two cases:
\begin{itemize}[leftmargin=1cm]
\item[$(a')$] If $ \dd_\Omega(x_0)\ge 4r^\theta$, then we set $R=2r^\theta$ so that
\[0 < r < r^\theta = \frac{R}{2} \le \frac{ \dd_\Omega(x_0)}{4},\]
hence case $(a)$ of the previous alternative occurs. By \eqref{ite2} we have
\beq\label{ite4}
\sigma(u,x_0,r) \le C\big[H^p_\alpha(u)+1\big]r^{N+p(\theta\alpha+(1-\theta)\delta)}.
\eeq
\item[$(b')$] If $ \dd_\Omega(x_0)<4r^\theta$, then we set $R=2r^\eta$ and (choosing $r$ small enough) we assume $R<\bar\rho/4$, so that
\[0 < r < R, \quad  \dd_\Omega(x_0) < 2R < \frac{\bar\rho}{2}\]
and case $(b)$ occurs. Therefore, by \eqref{ite3} and $r<1$ we have
\begin{align*}
\sigma(u,x_0,r) &\le CH_\alpha^p(u)\Big(\frac{r^\theta+4r^\theta}{2r^\eta}\Big)^{ps}r^{N+p\alpha\eta}+C\big[H_\alpha^{\xi_p}(u)+1\big]r^{\eta(N+\xi_p\alpha+p\gamma)} \\
\nonumber &\le CH_\alpha^p(u)r^{N+ps(\theta-\eta)+p\alpha\eta}+C\big[H_\alpha^{\xi_p}(u)+1\big]r^{\eta(N+\xi_p\alpha+p\gamma)}.
\end{align*}
Note that $\xi_p<p$, hence $H_\alpha^{\xi_p}(u)\le H_\alpha^p(u)+1$ and we obtained in this case
\begin{align}\label{ite5}
\sigma(u,x_0,r) &\le C\big[H_\alpha^p(u)+1\big]r^{N+ps(\theta-\eta)+p\alpha\eta} \\
\nonumber&+ C\big[H_\alpha^p(u)+1\big]r^{\eta(N+\xi_p\alpha+p\gamma)}.
\end{align}
\end{itemize}
In order to make powers of $r$ in \eqref{ite4} and \eqref{ite5}, respectively, agree, we seek $\eta$, $\theta$ s.t.\
\[N+p(\theta\alpha+(1-\theta)\delta) = N+ps(\theta-\eta)+p\alpha\eta = \eta(N+\xi_p\alpha+p\gamma).\]
A straightforward algebraic computation leads to the following expression (in which for simplicity we let $\theta$ depend on $\eta$):
\[\begin{split}
\eta &= \frac{N(s+\delta-\alpha)+ps\delta}{(N+\xi_p\alpha+p\gamma)(s+\delta-\alpha)+p(s-\alpha)(\delta-\alpha)}, \\
\theta &= \frac{\eta(N+ps+(\xi_p-p)\alpha+p\gamma)-N}{ps}.
\end{split}\]
We next check that such choice is admissible. From $\alpha<\delta$ it immediately follows that $\eta>0$. Besides, by definition of $\gamma$ and $\xi_p$ we have for all $p>1$
\beq\label{ite6}
p\gamma+\xi_p\alpha-p\alpha > 0.
\eeq
Indeed, the latter is equivalent in both cases $p\le 2$ and $p\ge 2$ to 
\[\alpha<p'\left(s-\frac{N}{pq}\right),\]
which again is granted by $\alpha<\delta$. The other required conditions on the parameters $\eta$, $\theta$ rephrase as follows:
\[\begin{split}
\eta < \theta \quad &\Longleftrightarrow \quad ps\delta(p\gamma+\xi_p\alpha-p\alpha) > 0,\\
\theta \le 1 \quad  &\Longleftrightarrow \quad (s-\alpha)(p\gamma+\xi_p\alpha-p\alpha) \ge 0.
\end{split}\]
Clearly, both inequalities follow from \eqref{ite6}. Therefore, we have $0<\eta<\theta\le 1$, as required. Now set for all $0\le\alpha\le\min\{s,\delta\}$
\[\phi(\alpha) = \theta\alpha+(1-\theta)\delta = \frac{N\alpha(s+\delta-\alpha)+s\delta(p\gamma+\xi_p\alpha)}{(N+p\gamma+\xi_p\alpha)(s+\delta-\alpha)+p(s-\alpha)(\delta-\alpha)},\]
so $N+p\phi(\alpha)$ equals all powers of $r$ in estimates \eqref{ite4}, \eqref{ite5}. Besides, we note that $\xi_p\le p$ for all $p>1$. Now, all the inequalities above summarize in the following variance estimate: for all $0\le\alpha<\min\{s,\delta\}$, $x_0\in\Omega$, and $r$ satisfying
\[0 < r < \Big(\frac{\bar\rho}{8}\Big)^\frac{1}{\eta} =: r_0,\]
we have
\beq\label{ite7}
\sigma(u,x_0,r) \le C\big[H_\alpha^p(u)+1\big]r^{N+p\phi(\alpha)}.
\eeq
Some further remarks on the continuous mapping $\phi:[0,\min\{s,\delta\}]\to\R$ (we include the supremum of the interval for simplicity) are now in order. First, for all $\alpha\ge 0$ we have
\[\phi(\alpha) > \alpha \quad \Longleftrightarrow \quad (p\gamma+\xi_p\alpha-p\alpha)(s-\alpha)(\delta-\alpha) > 0.\]
Recalling \eqref{ite6}, we deduce that $\phi(\alpha)>\alpha$ (in particular, $\phi(\alpha)>0$) for all $0\le\alpha<\min\{s,\delta\}$, while $\phi(\alpha)=\alpha$ at $\alpha=\min\{s,\delta\}$.
\vskip2pt
\noindent
By Theorem \ref{cam} and \eqref{ite7}, we have for all $0\le\alpha<\min\{s,\,\delta\}$
\begin{align*}
H_{\phi(\alpha)}^p(u) &= [u]_{C^{\phi(\alpha)}(\overline\Omega)}^p \\
&\le C_\alpha\sup_{x_0\in\Omega,\,0<r<r_0}\,\frac{\sigma(u,x_0,r)}{r^{N+p\phi(\alpha)}} \le C_\alpha\big[H_\alpha^p(u)+1\big].
\end{align*}
We next iterate on the basis of the previous inequality. Set $\alpha_0=0$, and for all $j\ge 0$ set $\alpha_{j+1}=\phi(\alpha_j)$. From the properties of $\phi$, it follows that $(\alpha_j)$ is an increasing sequence, tends to $\min\{s,\delta\}$ as $j\to\infty$, and for all $j\ge 0$
\[H_{\alpha_{j+1}}(u) \le C_j\big[H_{\alpha_j}(u)+1\big] \le \widetilde{C}_j,\]
with $C_j,\widetilde{C}_j>0$ depending on the data, $q$, and the index $j$, where we also have used that $H_0(u)$ is bounded by \eqref{ite1}.
\vskip2pt
\noindent
We can now conclude the proof. Fix $0<\alpha<\min\{s,\bar\alpha\}$. Then, we can find $\alpha<\delta<\bar\alpha$ depending only on the data, $q$, and $\alpha$, so that $\alpha<\min\{s,\delta\}$. Define the sequence $(\alpha_j)$ as above, then we can find an integer $j\ge 0$ s.t.\ $\alpha_{j+1}\ge\alpha$, and $H_{\alpha_{j+1}}(u)$ is bounded by a constant $\widetilde{C}_\alpha>0$ only depending on the data, $q$, and $\alpha$. Therefore, $u\in C^\alpha(\overline\Omega)$ and by \eqref{ite1} we have
\[\|u\|_{C^{\alpha_{j+1}}(\overline\Omega)} = \|u\|_{L^\infty(\Omega)}+[u]_{C^{\alpha_{j+1}}(\overline\Omega)} \le C_0+\widetilde{C}_\alpha.\]
By the continuous embedding $C^{\alpha_{j+1}}(\overline\Omega)\hookrightarrow C^\alpha(\overline\Omega)$, we conclude that $u\in C^\alpha(\overline\Omega)$ and $[u]_{C^\alpha(\overline\Omega)} \le C_\alpha$
with $C_\alpha>0$ depending on the data, $q$, and $\alpha$. Homogeneity of $\fpl$ finally allows to remove the restraint $\|f\|_{L^q(\Omega)}=1$, thus achieving the conclusion.
\end{proof}

\noindent
We can now prove our global regularity result:
\vskip4pt
\noindent
{\em Proof of Theorem \ref{hol}.} If $N/ps<q\le N/s$, then we have
\[p'\Big(s-\frac{N}{pq}\Big) \le s \ (< 1),\]
hence in particular $\bar\alpha\le s$. By Lemma \ref{ite}, for all $0\le\alpha\le s$ we have $u\in C^\alpha(\overline\Omega)$ and we can find $C_\alpha>0$ depending on the data, $q$, and $\alpha$, s.t.\
\[\|u\|_{C^\alpha(\overline\Omega)} \le C_\alpha\|f\|_{L^q(\Omega)}^\frac{1}{p-1}.\]
\vskip2pt
\noindent
We then consider case $q>N/s$, and by homogeneity we assume $\|f\|_{L^q(\Omega)}=1$. Note that now
\beq
\label{diss}
p'\Big(s-\frac{N}{pq}\Big) > s,
\eeq
that is, $\bar\alpha=s$. Therefore, Lemma \ref{ite} ensures $u\in C^\alpha(\overline\Omega)$ for all $0<\alpha<s$, with
\[\|u\|_{C^\alpha(\overline\Omega)} \le C_\alpha.\]
In order to reach the limit exponent $s$, we argue as in \cite[Theorem 2.7]{IM}. We apply Lemma \ref{rsb} with $\beta=s$, $\nu=0$. First, from Theorem \ref{fbr} it follows in particular that
\[\Big\|\frac{u}{\ds}\Big\|_{L^\infty(\Omega)} \le C,\]
hence we have for all $x\in\Omega$
\beq\label{hol3}
|u(x)| \le C\ds(x),
\eeq
with $C>0$ depending on the data, so that $\|u\|_{L^\infty(\Omega)}\le C{\rm diam}(\Omega)^s$ and hypothesis \ref{rsb1} of Lemma \ref{rsb} holds. Next, let $x_0\in\Omega$ and $ \dd_\Omega(x_0)=4R$. By Theorem \ref{lhr} with $\delta=s$, we have $u\in C^s(\overline{B}_{R/2}(x_0))$ and
\beq\label{hol4}
[u]_{C^s(\overline{B}_{R/2}(x_0))} \le \frac{C}{R^s}\left[\|u\|_{L^\infty(B_R(x_0))}+R^{p'(s-\frac{N}{pq})}+{\rm Tail}(u,x_0,R)\right].
\eeq
We now estimate the three terms on the right hand side of \eqref{hol4}. From \eqref{hol3} and $ \dd_\Omega(x_0)=4R$ we have
\[\|u\|_{L^\infty(B_R(x_0))} \le CR^s.\]
For the second term, thanks to \eqref{diss}, we have
\[R^{p'(s-\frac{N}{pq})} \le {\rm diam}(\Omega)^{p'(s-\frac{N}{pq})-s}R^s.\]
For the tail term, let $\tilde x\in\partial\Omega$ be a point (not necessarily unique) s.t.\ $ \dd_\Omega(x_0)=|x_0-\tilde x|$, then for all $x\in\R^N$ we have
\[ \dd_\Omega(x) \le |x-\tilde x| \le |x-x_0|+4R,\]
hence by \eqref{hol3} and subadditivity
\[|u(x)| \le C\big(|x-x_0|^s+R^s\big).\]
We deduce the following integral estimate:
\begin{align*}
\int_{B_R^c(x_0)}\frac{|u(x)|^{p-1}}{|x-x_0|^{N+ps}}\,dx &\le C\int_{B_R^c(x_0)}\frac{dx}{|x-x_0|^{N+s}}+CR^{(p-1)s}\int_{B_R^c(x_0)}\frac{dx}{|x-x_0|^{N+ps}} \\
&\le \frac{C}{R^s}+\frac{CR^{(p-1)s}}{R^{ps}} \le \frac{C}{R^s},
\end{align*}
which in turn, by \eqref{tail}, implies
\[{\rm Tail}(u,x_0,R) \le CR^s.\]
Plugging such estimates into \eqref{hol4}, we obtain
\[[u]_{C^s(\overline{B}_{R/2}(x_0))} \le C,\]
hence hypothesis \ref{rsb2} of Lemma \ref{rsb} is satisfied as well (with $\nu=0$). Finally, by \eqref{hol3} again, for all $\bar x\in\partial\Omega$ and $r>0$ small enough we have
\[\underset{D_r(\bar x)}{\rm osc}\,u \le 2\sup_{D_r(\bar x)}\,\ds \le Cr^s,\]
so hypothesis \ref{rsb3} of Lemma \ref{rsb} holds. Therefore, setting $\alpha=s$, we have $u\in C^s(\overline\Omega)$ and $[u]_{C^s(\overline\Omega)}\le C_s$ for some $C_s>0$ depending on the data and $q$. Invoking \eqref{hol3} once again to estimate $\|u\|_{L^\infty(\Omega)}$, we obtain
\[\|u\|_{C^s(\overline\Omega)} \le C_s,\]
for a possibly bigger $C_s>0$ depending on the data and $q$. As above, exploiting homogeneity of $\fpl$ we may remove the restraint $\|f\|_{L^q(\Omega)}=1$ and thus conclude the proof.
\qed

\appendix

\section{A multidimensional example}\label{app}

\noindent
We aim at extending Example \ref{opt} to any dimension $N\ge 2$, at least for $p=2$, $0<s<1$. Denote $x=(x',x_N)\in\R^N$ and set
\[\R^N_+ = \big\{x\in\R^N:\,x_N>0\big\}.\]
Fix $0<\alpha<s$ and set for all $x\in\R^N$
\[u(x) = |x|^{\alpha-s}(x_N)_+^s.\]
Clearly $u\in C^\infty(\R^N_+)$, but we have at most $u\in C^\alpha$ at the boundary $\{x_N=0\}$. Besides, $u=0$ in $\R^N\setminus\R^N_+$. We claim that
\begin{enumroman}
\item\label{app1} $u\in\widetilde{W}^{s,2}(B_R(0))$ for all $R>0$;
\item\label{app2} there exists $f$, lying in $L^q_{\rm loc}(\R^N_+)$ for any $1<q<N/(2s-\alpha)$, s.t.\ locally in $\R^N_+$
\[(-\Delta)^su = f.\]
\end{enumroman}
The range of $q$ in \ref{app2} is optimal, thus allowing for the same construction as in Example \ref{opt} for $N/2s<q<N/s$ by appropriately choosing $\alpha$.
\vskip2pt
\noindent
We first prove \ref{app1}. Set for all $x\in\R^N\setminus\{0\}$
\[g(x) = \int_{\R^N}\frac{|u(x)-u(y)|^2}{|x-y|^{N+2s}}\,dy.\]
First we show that $g(x)$ is well defined for all $|x|=2r>0$. Split the integration domain into $B_r(x)$ and $B_r^c(x)$. In $B_r(x)$, the function $u$ is Lipschitz with a constant $L>0$ depending on $s,\alpha,r$, so
\[\int_{B_r(x)}\frac{|u(x)-u(y)|^2}{|x-y|^{N+2s}}\,dy \le L^2\int_{B_r(x)}\frac{dy}{|x-y|^{N+2(s-1)}} < \infty.\]
Besiddes, by $\alpha$-H\"older continuity of $u$ and $u(0)=0$, for all $y\in B_r^c(x)$ we have
\[|u(y)| = \le |y|^\alpha \le (2r)^\alpha+|x-y|^\alpha,\]
which in turn implies
\begin{align*}
\int_{B_r^c(x)}\frac{|u(x)-u(y)|^2}{|x-y|^{N+2s}}\,dy &\le \int_{B_r^c(x)}\frac{[(2r)^\alpha+|x-y|^\alpha]^2}{|x-y|^{N+2s}}\,dy \\
&\le Cr^{2\alpha}\int_{B_r^c(x)}\frac{dy}{|x-y|^{N+2s}}+C\int_{B_r^c(x)}\frac{dy}{|x-y|^{N+2(s-\alpha)}}\,dy < \infty.
\end{align*}
We point out some further properties: from $\alpha$-positive homogeneity of $u$ it follows that $g$ is $2(\alpha-s)$-positively homogeneous. We next show that $g\in L^1(B_R(0))$ for all $R>0$, passing to spherical coordinates and setting $S_r=\partial B_r(0)$ for all $r>0$:
\begin{align*}
\int_{B_R(0)}g(x)\,dx &= \Big[\int_0^R r^{N-1+2(\alpha-s)}\,dr\Big]\,\Big[\int_{S_1}g(e)\,d\mathcal{H}^{N-1}\Big] \\
&= \frac{R^{N+2(\alpha-s)}}{N+2(\alpha-s)}\int_{S_1}g(e)\,d\mathcal{H}^{N-1} < \infty,
\end{align*}
where we also used the aforementioned homogeneity and boundedness of $g$ in $S_1$. By the above argument (with radius $2R$), we have
\[\iint_{B_{2R}(0)\times B_{2R}(0)}\frac{|u(x)-u(y)|^2}{|x-y|^{N+2s}}\,dx\,dy \le \int_{B_{2R}(0)}g(x)\,dx < \infty,\]
hence $u\in W^{s,2}(B_{2R}(0))$. Also, from $\alpha<s$ we get
\[\int_{\R^N}\frac{|u(x)|}{(1+|x|)^{N+2s}}\,dx \le \int_{\R^N}\frac{|x|^\alpha}{(1+|x|)^{N+2s}}\,dx < \infty.\]
Recalling the definition give in Section \ref{sec2}, we see that $u\in\widetilde{W}^{s,2}(B_R(0))$ and prove \ref{app1}.
\vskip2pt
\noindent
In order to prove \ref{app2}, we explicity compute the $s$-fractional Laplacian of $u$ at $x\in\R^N_+$, exploiting the regularity of $u$ is such domain (here P.V.\ stands for principal value):
\begin{align*}
(-\Delta)^su(x) &= {\rm P.V.}\,\int_{\R^N}\frac{|x|^{\alpha-s}(x_N)_+^s-|y|^{\alpha-s}(y_N)_+^s}{|x-y|^{N+2s}}\,dy \\
&= |x|^{\alpha-s}{\rm P.V.}\,\int_{\R^N}\frac{(x_N)_+^s-(y_N)_+^s}{|x-y|^{N+2s}}\,dy+{\rm P.V.}\,\int_{\R^N}\frac{|x|^{\alpha-s}-|y|^{\alpha-s}}{|x-y|^{N+2s}}(y_N)_+^s\,dy \\
&= {\rm P.V.}\,\int_{\R^N}\frac{|x|^{\alpha-s}-|y|^{\alpha-s}}{|x-y|^{N+2s}}(y_N)_+^s\,dy =: f(x),
\end{align*}
where we have used that $(x_N)_+^s$ is $s$-harmonic in $\R^N_+$. Next, we fix $1\le q<N/(2s-\alpha)$ and prove that $f\in L^q(B_R^+(0))$ for all $R>0$ (we set $B_R^+(0)=B_R(0)\cap\R^N_+$, and $S_R^+$ has the same meaning). In fact, it is seen as above that $f$ is $(\alpha-2s)$-positive homogeneous, so we are reduced to prove
\[\int_{B_1^+(0)}|f(x)|^q\,dx < \infty.\]
Again we use spherical coordinates and homogeneity:
\[\int_{B_1^+(0)}|f(x)|^q\,dx = \Big[\int_0^1 r^{N-1+q(\alpha-2s)}\,dr\Big]\,\Big[\int_{S_1^+}|f|^q\,d\mathcal{H}^{N-1}\Big].\]
The first integral is finite due to $N-1+q(\alpha-2s)>-1$. For the second, we will prove in fact more, i.e., that $f\in L^\infty(S_1^+)$. Fix $x\in S_1^+$ and split the integral as follows:
\begin{align*}
f(x) &= {\rm P.V.}\,\int_{B_{1/2}(x)}\frac{|x|^{\alpha-s}-|y|^{\alpha-s}}{|x-y|^{N+2s}}(y_N)_+^s\,dy+\int_{B_{1/2}^c(x)}\frac{|x|^{\alpha-s}-|y|^{\alpha-s}}{|x-y|^{N+2s}}(y_N)_+^s\,dy \\
&=: f_1(x)+f_2(x).
\end{align*}
We first deal with the non-singular term $f_2(x)$. For all $y\in B_{1/2}^c(x)$ we have $(y_N)_+\le |y|$ and
\[|y|^s+|y|^\alpha \le 2+|x-y|^s+|x-y|^\alpha,\]
hence
\begin{align*}
|f_2(x)| &\le \int_{B_{1/2}^c(x)}\frac{|y|^s+|y|^\alpha}{|x-y|^{N+2s}}\,dy \\
&\le C\int_{B_{1/2}^c(0)}\frac{dz}{|z|^{N+2s}}+C\int_{B_{1/2}^c(0)}\frac{dz}{|z|^{N+s}}+C\int_{B_{1/2}^c(0)}\frac{dz}{|z|^{N+2s-\alpha}} < \infty,
\end{align*}
and the latter is independent of $x$. Passing to $f_1$, we set for brevity $v(x)=|x|^{\alpha-s}$, which is $C^2$ in $\R^N\setminus\{0\}$, and rewrite the integral using the changes of variables $y=(x'\pm z',x_N+z_N)$, respectively, with $z\in B_{1/2}(0)$. Linearity of the operator (this is the main reason why we take $p=2$) allows to rephrase the integral via a second-order difference, thus removing the principal value:
\[f_1(x) = \frac{1}{2}\int_{B_{1/2}(0)}\frac{2v(x',x_N)-v(x'+z',x_N+z_N)-v(x'-z',x_N+z_N)}{|z|^{N+2s}}(x_N+z_N)_+^s\,dz.\]
Arguing as in \cite[Lemma 2.11]{IMS}, we have
\[\big|2v(x',x_N)-v(x'+z',x_N+z_N)-v(x'-z',x_N+z_N)\big| \le \|v\|_{C^2(B_2(0)\setminus B_{1/2}(0))}|z|^2,\]
hence
\begin{align*}
|f_1(x)| &\le \frac{1}{2}\int_{B_{1/2}(0)}\frac{\|v\|_{C^2(B_2(0)\setminus B_{1/2}(0))}|z|^2}{|z|^{N+2s}}(x_N+z_N)_+^s\,dz \\
&\le C\int_{B_{1/2}(0)}\frac{dz}{|z|^{N+2(s-1)}} < \infty,
\end{align*}
and the last quantity is independent of $x$. Thus, $f\in L^\infty(S_1^+)$, which implies \ref{app2}.

\vskip4pt
\noindent
{\bf Acknowledgement.} Both authors are members of GNAMPA (Gruppo Nazionale per l'Analisi Matematica, la Probabilit\`a e le loro Applicazioni) of INdAM (Istituto Nazionale di Alta Matematica 'Francesco Severi') and supported by the the research project {\em Problemi non locali di tipo stazionario ed evolutivo} (GNAMPA, CUP E53C23001670001). A.I.\ is partially supported by the research project {\em Partial Differential Equations and their Role in Understanding Natural Phenomena} (Fondazione di Sardegna 2023, CUP F23C25000080007). S.M.\ is partially supported by PRIN project 2022ZXZTN2 - {\em Nonlinear Differential Problems with Applications to Real Phenomena} and by project PIACERI Linea 1 - EdP.EReMo.

\end{document}